\documentclass{amsart}
\usepackage[utf8]{inputenc}
\usepackage{amsmath}
\usepackage{amssymb}
\usepackage{amsthm}
\usepackage{fullpage}

\usepackage{breqn}
\newtheorem{theorem}{Theorem}[]
\newtheorem{lemma}{Lemma}[]

{}
\newtheorem{proposition}{Proposition}{}
\title{High moments of theta functions and character sums}
\author{Barnabás Szabó }
\date{}
\usepackage[maxbibnames=99]{biblatex}
\numberwithin{equation}{section}

\address{Mathematics Institute, Zeeman Building, University of Warwick, Coventry CV4 7AL, England}
\email{\tt Barnabas.Szabo@warwick.ac.uk}
\begin{document}
\maketitle
\begin{abstract}
   Assuming the Generalised Riemann Hypothesis, we prove a sharp upper bound on moments of shifted Dirichlet $L$-functions. We use this to obtain conditional upper bounds on high moments of theta functions. Both of these results strengthen theorems of Munsch, who proved almost sharp upper bounds for these quantities. The main new ingredient of our proof comes from a paper of Harper, who showed the related result $\int_{0}^T |\zeta(1/2+it)|^{2k} \ll_k T(\log T)^{k^2} $ for all $k\geq 0$ under the Riemann Hypothesis. Finally, we obtain a sharp conditional upper bound on high moments of character sums of arbitrary length. 
\end{abstract}
\section{Introduction}
Calculating the moments of families of $L$-functions has been the subject of research amongst number theorists for many decades. Estimating these moments almost always comes down to approximating the $L$-function by the corresponding Dirichlet polynomial, which in turn can be shown to exhibit significant cancellation upon integrating. In 2008, Soundararajan \cite{sound} found an ingenious way to obtain an upper bound on the zeta function, conditional on the Riemann Hypothesis (RH). It essentially takes the form
\begin{equation}
\label{soundupper}
    \log |\zeta(1/2+it)|\leq \Re \sum_{n\leq x}\frac{\Lambda(n)}{n^{1/2+1/\log x+it} \log n}\Big(1-\frac{\log n}{\log x}\Big)+\frac{\log t}{\log x}+O(1),
\end{equation}
where $t>2$ and $2<x<t^2$. Soundararajan used this to derive almost sharp upper bounds for moments of $\zeta(s)$ on the critical line. In particular, he showed that under RH, for each fixed real $k\geq 0$ one has
$$\int_0^T |\zeta(1/2+it)|^{2k}dt \ll_{k,\epsilon} T(\log T)^{k^2+\epsilon}.$$
This upper bound was improved by Harper \cite{harpersharp} (under RH) to
$$\int_0^T |\zeta(1/2+it)|^{2k}dt \ll_k T(\log T)^{k^2},$$
and this is sharp up to a constant (see \cite{radziwill2013continuous} for the corresponding unconditional lower bound when $k\geq 1$). Harper started with \eqref{soundupper} as well, however by a careful analysis of large values of Dirichlet polynomials he managed to bound the $2k$-th moment of $\zeta(s)$ without losing more than a constant.

The upper bound \eqref{soundupper} can be generalised for many classes of $L$-functions, which can be used to derive moment inequalities. In \cite{munschshifted} Munsch proved that if the Generalised Riemann Hypothesis (GRH) holds, then 
\begin{equation}
  \label{munsch1}
\sum_{\chi\in X_q^*}|L(1/2+it_1,\chi) \cdot L(1/2+it_2, \chi)\cdots L(1/2+it_{2k},\chi)|\ll_{\epsilon, k} \phi(q) (\log q)^{k/2+\epsilon} \prod_{1\leq i<j\leq 2k} g^{1/2}(|t_i-t_j|).
\end{equation}
Here $X_q^*$ denotes the set of primitive Dirichlet characters mod $q$, $k$ is a positive integer, and the $t_j$ are real numbers that may grow slowly with $q$. Moreover, roughly speaking, $g:\mathbb{R}_{\geq 0}\rightarrow \mathbb{R}^+$ is a correlation factor which is decreasing and $g(0)=\log q$. In particular \eqref{munsch1} implies that
\begin{equation*}
    \sum_{\chi\in X_q^*} |L(1/2,\chi)|^{2k} \ll_{k,\epsilon} (\log q)^{k^2+\epsilon}.
\end{equation*}
If the $t_j$ are relatively far apart, then (\ref{munsch1}) becomes stronger, which is expected as the values of the $L$-functions `correlate' less with each other. 
In our first theorem we get rid of the $(\log q)^{\epsilon}$ factor in \eqref{munsch1} and also slightly improve upon the correlation function $g$. Our argument will be very similar to the way Harper improved Soundararajan's moment inequality.
\begin{theorem}
\label{t1}
Let $2k\geq 1$ be a fixed integer and $a_1,\ldots, a_{2k}, A$ be fixed positive real numbers. Assume that for any Dirichlet character $\chi$ mod $q$, the corresponding $L$-function $L(s,\chi)$ satisfies the Riemann Hypothesis. Let $X_q^*$ denote the set of primitive characters modulo $q$. Let $t=(t_1,\ldots ,t_{2k})$ be a real $2k$-tuple with $|t_j|\leq  q^A$. Then
$$\sum_{\chi \in X_q^*} \big| L\big(1/2+it_1,\chi \big) \big|^{a_1} \cdots \big| L\big(1/2+it_{2k},\chi \big) \big|^{a_{2k}}\ll \phi(q)(\log q)^{(a_1^2+\cdots +a_{2k}^2)/4} \prod_{1\leq i<j\leq 2k} g(|t_i-t_j|)^{a_ia_j/2},  $$
where $g:\mathbb{R}_{\geq 0} \rightarrow \mathbb{R}$ is the function defined by
$$g(x) =\begin{cases}
\log q  & \text{if } x\leq \frac{1}{\log q} \text{ or } x\geq e^q \\
\frac{1}{x} & \text{if }   \frac{1}{\log q}\leq x\leq 10, \\
\log \log x & \text{if } 10\leq x\leq e^q.
\end{cases}
$$
Here the implied constant depends on $k$, $A$ and the $a_j$ but not on $q$ or the $t_j$.
\end{theorem}
Recently, Curran \cite{curran2023correlation} has shown essentially the same type of upper bound for shifted moments of the Riemann zeta function, moreover his method of proof similar to ours. His theorem generalises and improves upon previous work by Chandee \cite{chandee2011correlation} and Ng, Shen and Wong \cite{ng2022shifted}.

Let us proceed to the topic of our second theorem. In \cite{munschshifted} Munsch used (\ref{munsch1}) to obtain (conditional) upper bounds on integer moments of $\theta$ functions.
For a Dirichlet character $\chi$ mod $q$ define $\kappa=\kappa(\chi)=(1-\chi(-1))/2$, i.e. $\kappa=1$ if $\chi$ is odd and $\kappa=0$ if $\chi$ is even. The $\theta$ function corresponding to $\chi$ is defined as 
$$\theta(x,\chi)=\sum_{n=1}^{\infty} \chi(n) n^{\kappa} e^{-\pi n^2 x/q}.$$
Let $X_q^+$ and $X_q^-$ denote the set of even and odd primitive Dirichlet characters respectively. In \cite{munschshifted} Munsch showed that for each fixed positive integer $k$ and $\epsilon>0$ one has
$$S_{2k}^+(q):=\sum_{\chi \in X_q^+ } |\theta(1,\chi)|^{2k}\ll_{k,\epsilon} \phi(q)q^{k/2}(\log q)^{(k-1)^2+\epsilon},$$
and
$$S_{2k}^-(q):=\sum_{\chi \in X_q^-} |\theta(1,\chi)|^{2k}\ll_{k, \epsilon} \phi(q)q^{3k/2}(\log q)^{(k-1)^2+\epsilon}.$$
In our next theorem we will remove the $\epsilon$ from the exponent using Theorem \ref{t1}. Moreover our result will hold for all real $k>2$, not just for integers.

\begin{theorem}
\label{t2}
Let $q>1$ be a positive integer,  and assume that for any Dirichlet character $\chi$ mod $q$, the corresponding $L$-function $L(s,\chi)$ satisfies the Riemann Hypothesis. Let $X_q^+$ and $X_q^-$ denote the set of even and odd primitive characters respectively. Let $k>2$ be a real number. Then
$$S_{2k}^+(q):=\sum_{\chi \in X_q^+ } |\theta(1,\chi)|^{2k}\ll_{k} \phi(q)q^{k/2}(\log q)^{(k-1)^2},$$
and
$$S_{2k}^-(q):=\sum_{\chi \in X_q^-} |\theta(1,\chi)|^{2k}\ll_{k} \phi(q)q^{3k/2}(\log q)^{(k-1)^2}.$$
\end{theorem}

These upper bounds are conjectured to be sharp up to a constant. In fact, when $k$ is an integer and $q$ is a prime, matching lower bounds were proven in \cite{munsch2016upper}. Our method breaks down when $k\leq 2$. The $k=2$ case has been studied before in \cite{louboutin2013second}, and an asymptotic formula was shown there. When $1<k<2$ it is expected that the same type of upper bound holds. When $0<k\leq 1$ and $q$ is a prime, Harper has recently shown (unconditionally) that 
\begin{equation*}
    S_{2k}^+(q)\ll \frac{\phi(q)q^{k/2}}{(1+(1-k)\sqrt{\log\log q} )^{k} },
\end{equation*}
and the same type of upper bound holds for $S_{2k}^-(q)$ with $q^{k/2}$ replaced by $q^{3k/2}$ (see Corollary 2 in \cite{harper2023typical}).

The reason why these upper bounds take this shape is related to the fact that character sums can be modeled by random multiplicative functions, the moments of which have been studied extensively by Harper in \cite{harper2020moments} and \cite{harper2020moments1}.

When $\chi$ is an even primitive character, the quantity $\theta(1,\chi)=\sum_{n=1}^{\infty} \chi(n) e^{-\pi n^2/q}$ behaves like $\approx \sum_{n\leq q^{1/2}} \chi(n)$. This is because when $1\leq n\leq q^{1/2}$, then $e^{-\pi n^2/q}\asymp 1$, and when $n>q^{1/2}$, the weight $e^{-\pi n^2/q}$ quickly decays to $0$. Therefore we expect that the character sum moment $\sum_{\chi\in X_q^*} \bigg|\sum_{n\leq q^{1/2} } \chi(n)\bigg|^{2k}$ can be upper bounded in a similar manner as $S^+_{2k}(q)$. This turns out to be true, moreover our methods are general enough that we do not need to restrict ourselves to sum up to $q^{1/2}$. We may consider the more general quantity
$$S_k(q,y):=\sum_{\chi\in X_q^*}\bigg|\sum_{n\leq y} \chi(n)\bigg|^{2k}.$$ 
Clearly it is enough to consider the case $2\leq y\leq q$. Theorem \ref{t2} suggests, when $k>2$ this quantity should be 
$$\ll \phi(q) y^k(\log y)^{(k-1)^2}.$$
We will prove this bound holds under GRH when $k>2$.
% In fact Harper \cite{harper2020moments} proves, a similar upper bound for large moments of Steinhaus random variables (see the mentioned paper for an introduction to this topic). 
When $q^{1/2}\leq y\leq q$ we can go slightly further and improve upon the logarithmic term. The Poisson summation formula for character sums suggests the relation $\big| \sum_{n\leq y} \chi(n)\big|\approx \frac{y }{q^{1/2} } \big| \sum_{n\leq q/y} \bar\chi(n)\big|$. This allows us to replace the term $\log y$ with $\log 2q/y$ when  $q^{1/2}\leq y\leq q$.

\begin{theorem}
\label{t3}
Let $k>2$ be a fixed real number and $q$ a large integer. Assume that for any Dirichlet character $\chi$ mod $q$, the corresponding $L$-function $L(s,\chi)$ satisfies the Riemann Hypothesis. If $2\leq y\leq q^{1/2}$, then
$$S_k(q,y)\ll_k \phi(q)y^k(\log y)^{(k-1)^2}.$$
Moreover, if $q^{1/2}<y\leq q$, we have
$$S_k(q,y)\ll_k \phi(q)y^k\Big(\log \frac{2q}{y}\Big)^{(k-1)^2}.$$
\end{theorem}
Note that in his recent work \cite{harper2023typical}, when $0\leq k\leq 1$ and $q$ is prime, Harper gave the unconditional upper bound
\begin{equation*}
    S_k(q,y)\ll \frac{\phi(q)y^k}{(1+(1-k)\sqrt{10\log\log L} )^{k} },
\end{equation*}
where $L=\min\{ y, q/y\}$. This upper bound is conjecturally sharp up to a constant, however what happens when $1<k<2$ is still an open problem.

Finally, we mention two unconditional results in the direction of Theorem \ref{t3}. Firstly, Montgomery and Vaughan \cite{montgomery1979mean} showed
that for any real $k>0$ and $2\leq y\leq q$ one has
\begin{equation*}
    S_k(q,y)\ll_{k} \phi(q)q^k.
\end{equation*}
In fact their statement is stronger than this, we refer the interested reader to Theorem 1 of \cite{montgomery1979mean}. Note that this proves Theorem \ref{t3} when $y\gg q$, however loses its strength when $y$ gets smaller.

Secondly, when $q$ is prime and $k$ is a positive integer, Cochrane and Zheng \cite{cochrane1998high} showed that for any $\epsilon>0$ and $2\leq y\leq q$ one has
\begin{equation*}
    S_k(q,y)\ll_{k,\epsilon} \phi(q) \big(q^{k-1+\epsilon}+y^kq^{\epsilon}\big).
\end{equation*}
When $y\gg q^{1-1/k}$, this is only worse than Theorem \ref{t3} by a factor of $q^{\epsilon}$.
\section{Overview of the proofs}
\noindent In this section we give a quick overview of the proof of each of the three theorems.

\subsection{Theorem 1}
Let $\chi$ be a primitive Dirichlet character mod $q$. We start with Lemma 2.3. of \cite{munschshifted}, which is a generalisation of \eqref{soundupper}, and roughly speaking can be written as
$$ \log|L(1/2+it,\chi)| \lessapprox \Re \sum_{p\leq x}\frac{\chi(p)}{p^{1/2+it} }+\frac{\log q}{\log x},$$
for any $2\leq x\leq t^2$ and $t=q^{O(1)}$. We may write
\begin{equation}
\label{explain}    
\begin{split}
 \big| L\big(1/2+it_1,\chi \big) \big|^{a_1} \cdots \big| L\big(1/2+it_{2k},\chi \big) \big|^{a_{2k}} &=\exp\bigg(\Re \sum_{j=1}^{2k}a_j \log |L(1/2+it_j, \chi)| \bigg) \\
 & \lessapprox \exp\bigg(\Re \sum_{j=1}^{2k}a_j \sum_{p\leq x} \frac{\chi(p) }{p^{1/2+it_j} }+O\Big(\frac{\log q}{\log x}\Big) \bigg) \\
 & = \exp^2\bigg(\Re \sum_{p\leq x} \frac{h(p)\chi(p) }{p^{1/2} }+O\Big(\frac{\log q}{\log x}\Big) \bigg), \\
\end{split}
\end{equation}
where $h(p)=\frac{1}{2}(a_1 p^{-it_1}+\cdots +a_{2k} p^{-it_{2k} })$. Choose parameters $q^{1/(\log \log q)^2}=x_0<x_1<\ldots <x_{\mathcal{I}}=x^{\epsilon}$, where $\epsilon$ is a small but fixed constant and $x_{i+1}=x_i^{20}$. For any $1\leq i\leq \mathcal{I}$ and $\chi$ mod $q$ let us define
\begin{equation*}
   D(i,\chi):= \Re \sum_{x_{i-1}<p\leq x_i } \frac{h(p)\chi(p) }{p^{1/2}}.
\end{equation*} 
Firstly, we handle the characters $\chi$, for which $D(i,\chi)$ is not too large for any $1\leq i\leq \mathcal{I}$, in other words, the corresponding Dirichlet polynomial behaves well. More precisely, let $\alpha_i=\big(\frac{\log q}{\log x_i}\big)^{3/4}$ and let
\begin{equation*}
    \mathcal{T}=\{\chi\in X_q:|D(i,\chi)|\leq \alpha_i  \text{ for each } 1\leq i\leq \mathcal{I} \}.
\end{equation*}
% $$\Re \sum_{p\leq x_I} \frac{h(p)\chi(p) }{p^{1/2}}= \sum_{j=1}^J \Re \sum_{x_{j-1}<p\leq x_j } \frac{h(p)\chi(p) }{p^{1/2}}=:D_1+\cdots +D_J.$$
For any $\chi\in \mathcal{T}$, we choose  $x=x_{\mathcal{I}}$ in \eqref{explain}. With this choice we have $\log q/\log x \ll 1$. On the other hand, as $D(i,\chi)\leq \alpha_i$, we may truncate the infinite series expansion
\begin{equation*}
    e^{D(i,\chi)}=\sum_{n=0}^{\infty}\frac{D(i,\chi)^n}{n!}
\end{equation*}
at $n=\lfloor 100\alpha_i\rfloor$ with a negligible error term. So we may write
\begin{equation*}
\sum_{\chi\in \mathcal{T}} \exp^2\bigg(\Re \sum_{p\leq x} \frac{h(p)\chi(p) }{p^{1/2} }+\frac{\log q}{\log x} \bigg)
\lessapprox \sum_{\chi \in \mathcal{T}} \prod_{i=1}^{\mathcal{I}} \bigg( \sum_{0\leq n\leq 100\alpha_i} \frac{D(i,\chi)^n}{n!}\bigg)^2\leq  \sum_{\chi \in X_q} \prod_{i=1}^{\mathcal{I}} \bigg( \sum_{0\leq n\leq 100\alpha_i} \frac{D(i,\chi)^n}{n!}\bigg)^2
\end{equation*}
The crucial point is that the expression on the right hand side is a Dirichlet polynomial, whose length is less than $q$ by the choice of $\alpha_i$. We now may swap the order of summation and use the orthogonality of characters to obtain significant cancellation. After a lengthy calculation, at the end of the argument we need to upper bound expressions of the shape
$$\sum_{p\leq x} \frac{|h(p)|^2}{p}=\sum_{p\leq x} \frac{1}{p} \bigg(\frac{1}{4}(a_1^2+\cdots +a_{2k}^2 ) +\sum_{1\leq i<j\leq 2k} \frac{a_ia_j}{2}\cos( |t_i-t_j| \log p)\bigg).$$
We can use Mertens's estimate and properties of the zeta function to get the bound
$$\sum_{p\leq q}\frac{\cos(\alpha \log p )}{p}\leq \log g(\alpha)+O(1),$$
which allows us to handle the contribution from $\chi\in \mathcal{T}$ (here recall the definition of $g(\alpha)$ from the statement of Theorem \ref{t1}).

If $\chi \not\in \mathcal{T}$, then there is some $1\leq i\leq \mathcal{I}$, for which $D(i,\chi)$ is large, but for each $1\leq j\leq i-1$, $D(j, \chi)$ is small. In this case we repeat a similar argument, but choose our cutoff parameter at $x=x_{i-1}$, so our Dirichlet polynomial behaves well. Now the problem is that $\frac{\log q}{\log x}$ is not a bounded quantity anymore. however we can obtain extra saving using the fact that $D(i,\chi)$ is large, which heuristically should only happen for few $\chi$.
\subsection{Theorem 2}

We now outline how to deduce Theorem \ref{t2} from Theorem \ref{t1}. By the theory of Mellin transforms, for any even primitive $\chi\in X_q^+$ and $c>0$ we can write
$$\theta(1,\chi)=\frac{1}{2\pi i} \int_{(c)} L(2s, \chi) \bigg(\frac{q}{\pi}\bigg)^{s} \Gamma(s)ds,$$
 where $(c)$ denotes the straight line contour from $c-i\infty$ to $c+i\infty$. The integral is absolutely convergent because of the exponential decay of $\Gamma(s)$ as $\Im s\to \infty$. We shift the line of integration to $c=1/4$. 
 
 For the moment, assume that $2k$ is an integer, in which case we may write the $2k$-th power of an integral as a $2k$-fold integral, so we obtain
$$\sum_{\chi \in X_q^+} |\theta(1,\chi)|^{2k}\ll q^{k/2} \int_{\mathbb{R}^{2k}}\sum_{\chi\in X_q^*}  \prod_{j=1}^{2k}|L(1/2+it_j,\chi)\Gamma(1/4+it_j/2)| dt_1\ldots dt_{2k}.$$

Now the exponential decay of $\Gamma(1/4+it/2)$ allows us to restrict our attention to the region where the $t_j$ are small, say bounded. Moreover, by Theorem \ref{t1} we know that $\sum_{\chi\in X_q^*} \prod_{j=1}^{2k} |L(1/2+it_j,\chi)| $ is small unless the $t_j$ are close to each other. The most important case is when the $t_j$ are at most $1/\log q$ apart, in which case the expression inside the integral can be as large as $\phi(q)(\log q)^{k^2}$. However the region in which this happens (assuming the $t_j$ are bounded) has volume $\ll 1/(\log q)^{2k-1}$ which gives us the main contribution of size $(\log q)^{k^2}/(\log q)^{2k-1}=(\log q)^{(k-1)^2}$. We can handle the remaining region by appropriately splitting up the integral into regions and apply Theorem 1 separately on the integrand in each region. 

This argument works if $2k$ is an integer, however we want a proof for all real $k>2$. We certainly have
\begin{equation*}
    \sum_{\chi \in X_q^+} |\theta(1,\chi)|^{2k}\ll q^{k/2}\sum_{\chi \in X_q^*} \bigg( \int_{-\infty}^{\infty} |L(1/2+it,\chi)\Gamma(1/4+it/2)|dt\bigg)^{2k},
\end{equation*}
and we would like to pull the $2k$-th power inside the integral using Hölder's inequality to arrive at similar expressions to the integer case. Applying  Hölder straight away is not sufficient, so
following a strategy outlined in a paper by Harper (see \cite{harper2020moments}, page 8), we pull out 3 copies of the $2k$-th power (note $k>2$)
\begin{align*}
    & \sum_{\chi \in X_q^*} \bigg( \int_{-\infty}^{\infty} |L(1/2+it,\chi)\Gamma(1/4+it/2)|dt\bigg)^{2k}\\
    = & \sum_{\chi \in X_q^*} \bigg(\int_{-\infty}^{\infty} |L(1/2+it,\chi)\Gamma(1/4+it/2)|dt\bigg)^3\cdot \bigg( \int_{-\infty}^{\infty} |L(1/2+it,\chi)\Gamma(1/4+it/2)|dt\bigg)^{2k-3},
\end{align*}
and apply Hölder to the $(2k-3)$-th power. Eventually, we will need to obtain a suitable upper bound on
\begin{equation}
\label{hellom}
    \int_{[0,1]^4} \sum_{\chi\in X_q^*} | L(1/2+it_1,\chi) L(1/2+it_2,\chi) L(1/2+it_3,\chi) L(1/2+iu,\chi)^{2k-3}|d\mathbf{t}du,
\end{equation}
which is possible by applying Theorem \ref{t1} with $a_1=a_2=a_3=1$ and $a_4=2k-3$. We have $\frac{1}{4}(a_1^2+a_2^2+a_3^2+a_4^2)=k^2-3k+3$, so the best upper bound we can get on \eqref{hellom} is $\phi(q)(\log q)^{k^2-3k+3}$. This means that unless $k^2-3k+3\leq (k-1)^2$, i.e. $k\geq 2$ we cannot use Theorem 1 to prove Theorem 2. On the other hand, when $k>2$ this argument works, by breaking up the integral into regions where the distances $|t_1-u|$, $|t_2-u|$ and $|t_3-u|$ do not change by more than a constant factor and applying Theorem \ref{t1} on the integrand for each region. Summing up all the contributions, we obtain Theorem 2.

Finally, we remark that pulling out only 2 copies would mean that the best bound we can get on \eqref{hellom} is $\phi(q)(\log q)^{\frac{1}{4}(1+1+(2k-2)^2)} =\phi(q)(\log q)^{k^2-2k+3/2}$, but $k^2-2k+3/2> (k-1)^2$. So we do need to pull out 3 copies at least to make the argument work. On the other hand, pulling out 4 copies would already require $2k\geq 4$, so it would not improve the range of $k$, for which the argument works.
% Notice that if we only pull out 2 copies, then with $a_1=a_2=1$ and $a_3=2k-2$, we have $\frac{1}{4}(a_1^2+a_2^2+a_3^2)=k^2-2k+3/2>(k-1)^2$ so we need to pull out at least 3 copies for the argument to work.
\subsection{Theorem 3}
We now turn to the proof of Theorem 3. Let $\chi$ be a character mod $q$. A similar approach to that of Theorem 2 starts with writing the character sum as a Perron integral
$$\sum_{n\leq y}\chi(n)=\int_{1+1/\log y-iT}^{1+1/\log y+iT}L(s,\chi)\frac{y^s}{s}ds+O\Big(\frac{y}{T}\log y\Big).$$
We shift the line of integration to $\Re s=1/2+1/ \log y$ instead of $\Re s=1/2$. This allows us to get better upper bounds on $|L(s,\chi)|$ while  $y^s\ll y^{1/2}$, so there we get a $\log y$ term is Theorem 3 instead of $\log q$. This is because on the line $\Re s=1/2+1/\log y$ we are able to deduce a slightly stronger version of Theorem 1, since the approximating Dirichlet polynomials may be taken to be shorter. Using the residue theorem and ignoring the contribution from the horizontal integrals we get 
$$\sum_{\chi \in X_q^*} \bigg| \sum_{n\leq y}\chi(n)\bigg|^{2k}\ll y^k \sum_{\chi \in X_q^*} \bigg( \int_{1/2+1/\log y-iT}^{1/2+1/\log y+iT}\frac{|L(s,\chi)|}{|s|} |ds|\bigg)^{2k}+\phi(q)\cdot \Big(\frac{y}{T}\log y\Big)^{2k}.$$
Because of the error term, we need essentially $T\geq y^{1/2}$. But then the $1/|s|$ term in the integral will contribute $\gg (\log y)^{2k}$, which is not acceptable if $2k>(k-1)^2$, i.e. $k<\sqrt3+2$. We want Theorem 3 for all $k>2$ so we need to tweak our argument. Firstly we will consider the weighted sum $\sum_{n}f(n)\chi(n)$, where $f$ is a continuous linear weight function defined as follows. Let $y_0=y-y/(\log y)^C$ for some large $C$. Define
\begin{equation}
f(n)=
\begin{cases}
    1                     & \text{ if } 1\leq n \leq y_0  \\
    1-\frac{n-y_0}{y-y_0} & \text{ if } y_0\leq n\leq y   \\
    0                     & \text{ otherwise}             \\
\end{cases}
\end{equation}

 The continuity of $f$ enables us to write the weighted sum as an absolutely convergent Perron integral, which solves the issue mentioned above. Moreover $y-y_0$ is small, so it is not hard to show that the weighted sum is ``close" to the original one. Handling the $2k$-th moment of the resulting integrals is done similarly to the proof of Theorem \ref{t2}.

The second part Theorem \ref{t3} has similar proof, but has an additional ingredient. As in the first part, we switch to the weighted version $\sum_{n\leq q/y} f(n)\chi(n)$ as well. However, we will shift the line of integration to $\Re s=1/2-1/\log y$ instead and use the inequality $|L(s,\chi)|\ll (tq)^{1/2-\sigma}|L(1-s,\bar\chi)|$, which is a consequence of the functional equation for Dirichlet $L$-functions. When showing that the weighted version is not far from the original one, we need to prove a crude bound on the moment of $\sum_{n\leq q/y} \chi(n)$. This is done using the Poisson summation formula for character sums discovered by Pólya in \cite{pohlya1918ugber}, and the proof is inspired by a paper of Montgomery and Vaughan \cite{montgomery1979mean}, where the authors show that for each $k>0$ one has 
\begin{equation*}
 \sum_{\chi\neq \chi_0} \max_{1\leq y\leq q} \bigg| \sum_{1\leq n\leq y} \chi(n)\bigg|^{2k}\ll \phi(q)q^k.
\end{equation*}
\section{A crude bound on shifted moments}
Before we start proving our main theorems, we need a crude upper bound on shifted moments of $L$-functions both on and off the critical line. This short section is devoted to the proof of the following proposition.
\begin{proposition}
\label{crude_prop}
    Let $y\geq 2$ be a real number and let the assumptions of Theorem \ref{t1} hold with the same notation. We then have
\begin{equation*}
    \sum_{\chi \in X_q^*} \big| L\big(1/2+1/\log y+it_1,\chi \big) \big|^{a_1} \cdots \big| L\big(1/2+1/\log y+it_{2k},\chi \big). \big|^{a_{2k}}\leq \phi(q)(\min \{\log y+1,\log q \})^{O(1)}.
\end{equation*}
Here the implied constant depends on $k$, $A$ (recall $|t_j|\leq t^A$) and the $a_j$, but not on $q$, $y$ or the $t_j$.
\end{proposition}
Note that taking $y\to\infty$, by the continuity of $L(s,\chi)$ we get the corresponding moment bound on the critical line.

For the proof of Proposition \ref{crude_prop} we need to make use of the following lemma, which is essentially Proposition 2.3 in \cite{munschshifted}.
\begin{lemma}
\label{up}
Let $y\geq 2$ and $t$ be real numbers and define $\log^+ t=\max\{ 0,\log t\}$. For any $2\leq x\leq q$ and $\chi\in X_q^*$ we have 
\begin{equation*}
    \label{variant}
    \log |L(1/2+1/\log y+it, \chi)| \leq \Re \sum_{n\leq x}\frac{\chi(n)\Lambda(n)}{n^{1/2+\max(1/\log y, 1/\log x) +it} \log n}\frac{\log x/n}{\log x}+\frac{\log q+\log^+ t}{\log x}+O(1/\log x).
\end{equation*}
\end{lemma}
\begin{proof}
    If $y<x$, we use equation 2.8. in \cite{munschshifted}. We substitute $s_0=1/2+1/\log y+it$ and note that the terms involving $F_{\chi}(s_0)$ have negative contribution, since $F_{\chi}(s_0)>0$ and $y<x$. If $y\geq x$, then in Proposition 2.3 of \cite{munschshifted} we substitute $\lambda=1$ and $\sigma=1/2+1/\log y$.
\end{proof}
\begin{proof}[Proof of Proposition \ref{crude_prop}]
Let us denote $L_0:=\min \{\log y+1, \log q\}$. Note that by Hölder's inequality, since we are allowed to lose a power of $L_0$ in our estimates, it is enough to show that for any fixed integer $k\geq 1$ and $|t|\leq q^A$ we have
\begin{equation*}
    \sum_{\chi \in X_q^*} \big| L\big(1/2+1/\log y+ it,\chi \big) \big|^{2k}\leq \phi(q) L_0^{O(1)}. 
\end{equation*}
Here, and throughout the proof, our constants are allowed to be dependent on $k$ and $A$.

By Lemma \ref{up} and Mertens' estimates, there are absolute constants $C_1\geq 1$ and $C_2\geq 1$, such that for any $2\leq x\leq q$ we have
\begin{equation}
\label{simplified}
    \log |L(1/2+1/\log y+it,\chi)|\leq \Re \sum_{p\leq x} \frac{\chi(p)}{p^{1/2+\max(1/\log y, 1/\log x)+it} }\frac{\log x/p}{\log x}+C_1\frac{\log q}{\log x}+C_2\log L_0.
\end{equation}
For any integer $V\geq 1$, let $N(V)$ be the number of characters $\chi \in X_q^*$, such that $\log |L(1/2+1/\log y+it,\chi)|\geq V$. Assume $\chi$ is counted by $N(V)$ and also that $V\geq 4C_2\log L_0$. Let us choose $x=q^{10C_1/V}$ in \eqref{simplified}, so for such characters, we have
\begin{equation}
\label{Vhalf}
    \Re \sum_{p\leq x} \frac{\chi(p)}{p^{1/2+\max(1/\log y,1/\log x)+it} }\frac{\log x/p}{\log x}\geq \frac{V}{2}.
\end{equation}
We now use Lemma 2.8. in \cite{munschshifted}, with $x=q^{10C_1/V}$, $k_1=\lfloor\frac{V}{100C_1}\rfloor$ and $a(p)=p^{-\max(1/\log y ,1/\log x)}\cdot \frac{\log x/p}{\log x}$ (note that we use $k_1$ because the variable $k$ has already been used before). We have $x^{k_1}\leq q^{1/10}$, so the lemma is applicable. Since $x\leq q$, we have
\begin{equation*}
    \sum_{p\leq x} \frac{|a(p)|^2}{p}\leq \sum_{p\leq x} \frac{1}{p^{1+1/\log y} }\leq \log L_0 +C_3,
\end{equation*}
for some constant $C_3\geq 1$, so Lemma 2.8. of \cite{munschshifted} and \eqref{Vhalf} implies
\begin{equation*}
    N(V)\cdot \Big(\frac{V}{2} \Big)^{2k_1}\leq \phi(q)k_1! (\log L_0+C_3)^{k_1},
\end{equation*}
 We use the crude bound $k_1!\leq V^{k_1}$. If $V\geq e^{10000C_1k}(\log L_0+C_3)$, we deduce
\begin{equation}
\label{markov}
    N(V)\leq \phi(q) e^{-4kV}.
\end{equation}
Finally, let $V_0=\big\lceil \max \{ e^{10000C_1k}(\log L_0+C_3), 4C_2\log L_0\}\big\rceil $, so if $V\geq V_0$ then \eqref{markov} holds. Therefore, by partial summation we get
\begin{equation*}
    \sum_{\chi \in X_q^*} \big| L\big(1/2+1/\log y+ it,\chi \big) \big|^{2k}\leq \phi(q)e^{2kV_0}+\sum_{V= V_0}^{\infty} N(V)e^{(V+1)2k}\leq  \phi(q)e^{2kV_0}+\phi(q) e^{2k}\sum_{V=V_0}^{\infty} e^{-2kV}\leq \phi(q) L_0^{O(1)},
\end{equation*}
which proves the proposition.
\end{proof}

\section{Proof of Theorem \ref{t1}}
\noindent We start with a lemma which gives an upper bound on
\begin{equation*}
    \Re \sum_{n\leq x} \frac{1}{p^{1+i\alpha}},
\end{equation*}
for any $x\geq 2$ and $\alpha\geq 0$.
Our lemma and its proof are based on the unconditional Lemma 2.9 in \cite{munschshifted}. As we are allowed to assume RH, we are able to strengthen it slightly.
% which refers to (2.1.6.) p.57 of \cite{granville2014multiplicative}.
% The statements there are proved unconditionally. In this paper we are allowed to use RH so we can slightly strengthen the lemma (even though it will not make much of a difference in the statement of Theorem \ref{t1}). 
\begin{lemma}
\label{mertenstype}
Let $\alpha>0$, then
\begin{equation*}
\sum_{p\leq x} \frac{\cos(\alpha \log p) }{p}\leq 
\begin{cases}
\log\log x+O(1)            & \text{if }  \alpha\leq 1/\log x \text{ or } \alpha\geq e^x         \\
\log(1/\alpha)+O(1)        & \text{if }  1/\log x\leq \alpha \leq 10,  \\
 \log\log\log \alpha+O(1) & \text{if }  10\leq \alpha\leq e^{x},                \\
\end{cases}
\end{equation*}
where the first two estimates are unconditional and the third one holds under RH. 
% Moreover (?) if $\log x \leq \alpha \leq \exp(x^{1/9})$, then
% $$\sum_{p\leq x} \frac{\cos(\alpha \log p) }{p}\ll 1$$
\end{lemma}
\begin{proof}
The first part is implied by Mertens' second estimate, since for any $\alpha$ we have
$$\sum_{p\leq x} \frac{\cos(\alpha \log p) }{p}\leq \sum_{p\leq x} \frac{1 }{p}\leq \log\log x+O(1).$$
By partial summation we obtain
\begin{equation*}
     \sum_{p\leq x} \frac{\cos(\alpha \log p) }{p}=\Re \sum_{p\leq x} \frac{1 }{p^{1+i\alpha }}= \Re \log \zeta(1+1/\log x +i\alpha)+O(1) =  \log |\zeta(1+1/\log x +i\alpha)|+O(1).
\end{equation*}

$\zeta(s)$ has a simple pole at $1$ with residue $1$, so if $1/\log x\leq \alpha\leq 10$, we have
$$|\zeta(1+1/\log x+i\alpha)|=\frac{1}{|1/\log x+i\alpha|} +O(1)\ll \frac{1}{\alpha},$$
which implies
$$\log |\zeta(1+1/\log x +i\alpha)|\leq \log(1/\alpha)+O(1), $$
which proves the second part. In the range $10\leq \alpha\leq e^x$, assuming RH we may use Corollary 13.16 in \cite{montgomery2007multiplicative} to get
$$\log|\zeta(1+1/\log x+i\alpha)|\leq \log\log\log \alpha+O(1),$$
which proves the third part.
\end{proof}
Our next lemma essentially restates Lemma \ref{up}, it gives an upper bound  on $\log |L(1/2+it, \chi)|$ in terms of a Dirichlet polynomial and an extra term which is easy to understand.
\begin{lemma}
\label{Lupper}
Let $\chi$ be a primitive character mod $q$, where $q>1$, let $T>0$ and $x\geq 2$. Define $\log^+ T=\max\{0,\log T\}$. Assuming the Generalised Riemann Hypothesis on $L(s,\chi)$, for $|t|\leq T$ uniformly one has 
\begin{equation}
\label{upperbound}
\log |L(1/2+it, \chi)| \leq \Re \sum_{n\leq x}\frac{\chi(n)\Lambda(n)}{n^{1/2+1/\log x +it} \log n}\frac{\log x/n}{\log x}+\frac{\log q+\log^+ T}{\log x}+O(1/\log x).
\end{equation}
\end{lemma}
\begin{proof}
This is Proposition 2.3. in \cite{munschshifted} with $\lambda=1$ and $\sigma=1/2$.
% We use the convention $s=\sigma+it$ and $\rho=\beta+i\gamma$, where $\rho$ is a non-trivial zero of $L(s,\chi)$. In particular GRH implies that $\beta=1/2$. 
% $$F_{\chi}(s)=\Re \sum_{\rho} \frac{1}{s-\rho}=\sum_{\gamma}\frac{\sigma-1/2}{(\sigma-1/2)^2+(t-\gamma)^2 },$$
% so $F_{\chi}(s)$ is non-negative for $\sigma\geq 1/2$.
% Using Hadamard's factorisation and Stirling's formula
% $$\Re -\frac{L'}{L}(s,\chi)\leq \frac{\log q+\log^+T}{2}-F_{\chi}(s)\leq \frac{\log q+\log^+T}{2} .$$
% Hence
% $$\log |L(1/2+it,\chi)|-\log|L(\sigma+it,\chi)|=\int_{1/2}^{\sigma}\Re -\frac{L'}{L}(s,\chi)ds\leq (\sigma-1/2)\bigg(\frac{\log q+\log T}{2}+O(1)\bigg)$$

% By Lemma 1 we also have
% $$\log|L(\sigma+it,\chi)|=\int_{\sigma}^{\infty} \Re -\frac{L'}{L}(s,\chi)ds=$$
% $$\int_{\sigma}^{\infty} \bigg(\sum_{n\leq x}\frac{\chi(n)\Lambda(n)}{n^{s}}\frac{\log x/n}{\log x}+\frac{1}{\log x}\bigg(\frac{L'}{L}(s,\chi)\bigg)'+\frac{1}{\log x}\sum_{n}\frac{x^{-2n-\kappa-s}}{(2n+\kappa+s)^2}+\frac{1}{\log x}\sum_{\rho}\frac{x^{\rho-s}}{(\rho-s)^2}\bigg)ds\leq$$

% $$\sum_{n\leq x}\frac{\chi(n)\Lambda(n)}{n^{\sigma+it}\log n}\frac{\log x/n}{\log x}+\frac{\log q+\log^+ T}{2\log x}+\frac{F_{\chi}(s)}{\log x}\bigg(\frac{x^{1/2-\sigma}}{\log x}\frac{1}{\sigma-1/2}-1\bigg)+O(1/\log x)$$
\end{proof}
The main contribution in the Dirichlet polynomial comes from $n=p$ a prime and further non-negligible contribution comes from $n=p^2$. The contribution from higher prime powers is $O(1)$. The next proposition is an easy consequence of the above lemma.
\begin{proposition}
Let $2k$ be a positive integer and let $A,a_1,a_2,\ldots, a_{2k}$ be positive constants, $x\geq 2$. Let $a:=a_1+\cdots+ a_{2k}+10$. Let $q$ be a large modulus and assume that GRH holds for $L(s,\chi)$, where $\chi$ is a primitive character mod $q$. Let $t_1,\ldots, t_{2k}$ real numbers with $|t_i|\leq q^A$. For any integer $n$, let
$$h(n):=\frac{1}{2}(a_1n^{-it_1}+\cdots +a_{2k}n^{-it_{2k}}).$$
Then
\begin{multline}
\label{mainupper}
a_1\log |L(1/2+it_1,\chi)|+\cdots +a_{2k}\log |L(1/2+it_{2k},\chi)|  \\
    \leq 2\cdot\Re \sum_{p\leq x} \frac{h(p)\chi(p)}{p^{1/2+1/\log x}}\frac{\log x/p}{\log x}+\Re\sum_{p\leq x^{1/2}} \frac{h(p^2)\chi(p^2)}{p}+(A+1)a\frac{\log q}{\log x}+O(1).
\end{multline}
\end{proposition}

\begin{proof}
This is an immediate consequence of (\ref{upperbound}), since for each $1\leq j\leq 2k$ we have
$$\Re\sum_{n\leq x}\frac{\chi(n)\Lambda(n)}{n^{1/2+1/\log x +it_j} \log n}\frac{\log x/n}{\log x}=\Re \sum_{p\leq x}\frac{\chi(p)p^{-it_j} }{p^{1/2+1/\log x} }\frac{\log x/p}{\log x}+\frac{1}{2}\Re \sum_{p\leq x^{1/2}}\frac{\chi(p^2)p^{-2it_j} }{p}+O(1).$$

% $$\sum_{p\leq x^{1/2} } \frac{\chi(p^2)}{p^{1+2/\log x +2it_j}}-\frac{\chi(p^2)}{p^{1 +2it_j}}= -\int_{1}^{1+2/\log x } \sum_{p\leq x^{1/2}}\frac{\chi (p^2)}{p^{2it_j} } \cdot \frac{\log p}{p^u}du \ll \frac{1}{\log x} \sum_{p\leq x^{1/2}} \frac{\log p}{p}\ll 1$$
\end{proof}

% We want to use (\ref{upperbound}) to show Theorem \ref{t1}, so we cannot lose more than a constant in our calculation. Thus we will take $x$ to be some small specified power of $q$ and $T\leq q^A$, so
% $$\frac{\log q+\log^+ T}{\log x}\ll 1.$$
 We now introduce a few definitions which will enable us to bound the Dirichlet polynomial in an effective way. Our way of treating the problem comes from \cite{harpersharp}.

Define $\beta_0=0$, $\beta_i=\frac{20^{i-1}}{(\log\log q)^2}$ for $i\geq 1$, let $\mathcal{I}=1+\max \{i: \beta_i\leq e^{-10000a^2(A+1)}\}$. Note that $\beta_1<\beta_2<\ldots <\beta_{\mathcal{I}}$ form a geometric progression, and $\beta_{\mathcal{I}}$ should be thought of as a small fixed constant. Recall that $h(p)=\frac{1}{2}(a_1p^{-it_1}+\cdots+a_{2k}p^{-it_{2k}})$. For any $1\leq i\leq j\leq \mathcal{I}$ let 
$$G_{(i,j)}(\chi)=\sum_{q^{\beta_{i-1}}<p\leq q^{\beta_i}} \frac{\chi(p)h(p)}{p^{1/2+1/\beta_{j}\log q } } \frac{\log( q^{\beta_j }/p )}{\log( q^{\beta_j})}.$$
Here the motivation is that in (\ref{mainupper}) we choose $x=q^{\beta_j}$ and we cut up our main Dirichlet polynomial into smaller pieces where in $G_{(i,j)}(\chi)$ we sum over the primes between $q^{\beta_{i-1}}$ and $q^{\beta_i}$.
Define
$$\mathcal{T}=\{ \chi \in X_q^* : |\Re G_{(i,\mathcal{I})}( \chi )|\leq \beta_i^{-3/4}, \forall 1\leq i\leq \mathcal{I} \},$$
and for each $0\leq j<\mathcal{I}$ let
$$\mathcal{S}(j)=\{\chi \in X_q^*:|\Re G_{(i,l)}(\chi)|\leq \beta_{i}^{-3/4} \; \forall 1\leq i\leq j, \, \forall i\leq l\leq \mathcal{I}\, \text{but}\, |\Re G_{(j+1,l)}(\chi)|>\beta_{j+1}^{-3/4} \, \text{for some}\, j+1\leq l\leq \mathcal{I} \}.$$

We now state three lemmas which we prove in the next section. After the statements we show how they imply Theorem \ref{t1}.

\begin{lemma}
\label{mainfirst}
We have
$$\sum_{\chi \in \mathcal{T}}\exp^2\bigg(  \Re \sum_{p\leq q^{\beta_{\mathcal{I}}}}  \frac{\chi(p)h(p)}{p^{1/2+1/(\beta_{\mathcal{I}}\log q ) }} \frac{\log( q^{\beta_{\mathcal{I} } }/p )}{\log( q^{\beta_{\mathcal{I}}})}\bigg)\ll \phi(q)(\log q)^{(a_1^2+\cdots +a_{2k}^2)/4} \prod_{1\leq i<j\leq 2k} g(|t_i-t_j|)^{a_ia_j/2} $$ 
\end{lemma}

\begin{lemma}
\label{mainsecond}
We have $|\mathcal_{S}(0)|\ll qe^{-(\log\log q)^2}$ and for $1\leq j\leq \mathcal{I}-1$ we have
\begin{multline*}
    \sum_{\chi \in \mathcal{S}(j)}\exp^2\bigg( \Re \sum_{p\leq q^{\beta_j}} \frac{\chi(p)h(p)}{p^{1/2+1/(\beta_j\log q)}}\frac{\log(q^{\beta_j}/p)}{\log(q^{\beta_j}) }\bigg)\ll \\
 e^{-\beta_{j+1}^{-1} \log (\beta_{j+1}^{-1})/200 } \phi(q)(\log q)^{(a_1^2+\cdots +a_{2k}^2)/4} \prod_{1\leq i<j\leq 2k} g(|t_i-t_j|)^{a_ia_j/2} 
\end{multline*}
\end{lemma}

\begin{lemma}
\label{important}
The statements of the previous two lemmas remain true if we replace the Dirichlet polynomial by
$$\Re \sum_{p\leq q^{\beta_j}} \frac{\chi(p)h(p)}{p^{1/2+1/(\beta_j\log q)}}\frac{\log(q^{\beta_j}/p)}{\log(q^{\beta_j}) }+\frac{1}{2}\Re \sum_{p\leq q^{\beta_j /2}}\frac{\chi(p^2)h(p^2) }{p}.$$
\end{lemma}

We know show how these lemmas imply Theorem \ref{t1}. We have
$$X_q^*=\mathcal{T}\cup \bigcup_{j=0}^{\mathcal{I}-1} \mathcal{S}(j).$$
For simplicity let us denote
\begin{equation}
\label{Bdef}
    B:= \phi(q)(\log q)^{(a_1^2+\cdots +a_{2k}^2)/4} \prod_{1\leq i<j\leq 2k} g(|t_i-t_j|)^{a_ia_j/2}.
\end{equation}
By  Proposition 1 and Lemma \ref{important} we have
\begin{align*}
   & \sum_{\chi \in \mathcal{T}}  \big| L\big(1/2+it_1,\chi \big) \big|^{a_1} \cdots \big| L\big(1/2+it_{2k},\chi \big) \big|^{a_{2k}} \\
   = &\sum_{\chi \in \mathcal{T}} \exp\bigg(  a_1\log |L(1/2+it_1,\chi)|+\cdots +a_{2k}\log |L(1/2+it_{2k},\chi)| \bigg) \\
    \leq &  \sum_{\chi\in \mathcal{T} } \exp^2\bigg(\Re \sum_{p\leq q^{\beta_{\mathcal{I}} } }  \frac{\chi(p)h(p)}{p^{1/2+1/\beta_{\mathcal{I}}\log q }  } \frac{\log( q^{\beta_{\mathcal{I} } }/p )}{\log( q^{\beta_{\mathcal{I}}})}+\frac{1}{2}\Re\sum_{p\leq q^{\beta_{\mathcal{I}}/2 }} \frac{\chi(p^2)h(p^2) }{p }+(A+1) \frac{a}{\beta_{\mathcal{I}}} +O(1) \bigg) \\
    \ll & B \\
\end{align*}
For each $1\leq j\leq \mathcal{I}-1$ we have $\beta_{j+1}^{-1} \log (\beta_{j+1}^{-1})/200\geq 2a(A+1)/\beta_j$, so
\begin{align*}
     \sum_{j=1}^{\mathcal{I}-1} \sum_{\chi \in \mathcal{S}(j)}  \big| L\big(1/2+it_1,\chi \big) \big|^{a_1} \cdots \big| L\big(1/2+it_{2k},\chi \big) \big|^{a_{2k}} 
    \ll & B \sum_{j=1}^{\mathcal{I}-1}\exp\bigg( \frac{a}{\beta_j}(A+1) -\beta_{j+1}^{-1} \log (\beta_{j+1}^{-1})/200  \bigg) \\
    \ll &  B \sum_{j=1}^{\mathcal{I}-1}\exp\bigg(-\frac{a}{\beta_j}(A+1) \bigg) \\
    \ll & B \\
\end{align*}
as the $\beta_j$ form a geometric progression.

Using the bound on $|\mathcal{S}(0)|$ from Lemma \ref{mainsecond} and Proposition \ref{crude_prop}, by Cauchy-Schwarz we get
\begin{align*}
    \sum_{\chi \in \mathcal{S}(0)}  \big| L\big(1/2+it_1,\chi \big) \big|^{a_1} \cdots \big| L\big(1/2+it_{2k},\chi \big) \big|^{a_{2k}} & \leq |\mathcal{S}(0)|^{1/2} \bigg( \sum_{\chi \in \mathcal X_q^*}  \big| L\big(1/2+it_1,\chi \big) \big|^{2a_1} \cdots \big| L\big(1/2+it_{2k},\chi \big) 
\big|^{2a_{2k}}\bigg)^{1/2} \\
 & \ll  \phi(q)e^{-(\log\log q)^2/10}(\log q)^{O(1)}\ll B. \\
\end{align*}
Adding up the last three inequalities gives us Theorem \ref{t1}.

\section{Proof of the main lemmas}

\begin{proof}[Proof of Lemma \ref{mainfirst}]

For simplicity write
$$F_i( \chi):=G_{(i,\mathcal{I})}(\chi)=\sum_{q^{\beta_{i-1}}<p\leq q^{\beta_i}} \frac{\chi(p)h(p)}{p^{1/2+1/\beta_{\mathcal{I}}\log q } } \frac{\log( q^{\beta_{\mathcal{I} } }/p )}{\log( q^{\beta_{\mathcal{I}}})}.$$
By defintion, for every $\chi\in \mathcal{T}$ and $1\leq i\leq \mathcal{I}$, we have $|\Re F_i(\chi)|\leq \beta_i^{-3/4}$. Using Stirling's formula, if $\chi\in \mathcal{T}$, then
\begin{equation*}
    \exp(\Re F_i(\chi))=\sum_{0\leq j\leq 100\beta_i^{-3/4}}\frac{(\Re F_i(\chi))^j}{j!}+O\big(e^{-200\beta_i^{-3/4}}\big), 
\end{equation*}
say. Since $\exp(\Re F_i(\chi))\geq e^{-\beta_i^{-3/4}}$, we have
\begin{equation*}
    \exp(\Re F_i(\chi))=\big(1+O(e^{-100\beta_i^{-3/4}})\big)\sum_{0\leq j\leq 100\beta_i^{-3/4}}\frac{(\Re F_i(\chi))^j}{j!}.
\end{equation*}
Therefore
\begin{align*}
     \sum_{\chi \in \mathcal{T}}\exp^2\bigg(  \Re \sum_{p\leq q^{\beta_{\mathcal{I}}}}  \frac{\chi(p)h(p)}{p^{1/2+1/\beta_{\mathcal{I}}\log q  }} \frac{\log( q^{\beta_{\mathcal{I} } }/p )}{\log( q^{\beta_{\mathcal{I}}})}\bigg) 
    = &    \sum_{\chi\in \mathcal{T} } \exp^2\bigg( \sum_{1\leq i\leq \mathcal{I}} \Re F_i( \chi) \bigg) \\
\ll &  \sum_{\chi\in \mathcal{T} }  \prod_{1\leq i \leq \mathcal{I}} \bigg[\sum_{0\leq j \leq 100\beta_i^{-3/4}} \frac{(\Re F_i( \chi) )^j}{ j!}  \bigg]^2 \\
\leq & \sum_{\chi\in X_q }  \prod_{1\leq i \leq \mathcal{I}} \bigg[\sum_{0\leq j \leq 100\beta_i^{-3/4}} \frac{(\Re F_i( \chi) )^j}{ j!}  \bigg]^2. \\
\end{align*}
Notice that in the last line we are summing over all $\chi\in X_q$ instead of just $\chi \in \mathcal{T}$. This can be done since the square of a real number is always non-negative. This is why we have put our expressions in the form $\exp^2(\cdots)$ instead of $\exp(\cdots)$. Now we change the order of summation and have $\sum_{\chi \in X_q}$ as the innermost sum so we can use the orthogonality of characters to get cancellation. We find that the last expression above is equal to
\begin{equation}
\label{bigsum}
      \sum_{\Tilde{j},\Tilde{l}}\bigg( \prod_{1\leq i \leq \mathcal{I}} \frac{1}{j_i!}\frac{1}{l_i!} \bigg)\sum_{\Tilde{p},\Tilde{q}} C(\Tilde{p},\Tilde{q})\sum_{\chi \in X_q  } \prod_{1\leq i \leq \mathcal{I}}\prod_{\substack{1\leq r\leq j_i \\ 1\leq s\leq l_i}}   \big( \Re\chi(p_{i,r} ) h(p_{i,r} ) \big) \big( \Re\chi(q_{i,s} ) h(p_{i,s})  \big). 
\end{equation}

Here $\Tilde{j}=(j_1,\ldots, j_{\mathcal{I}})$ and $\Tilde{l}=(l_1,\ldots, l_{\mathcal{I}})$ are vectors of integers where $0\leq j_i,l_i\leq 100\beta_i^{-3/4}$. In addition $\Tilde{p}=(p_{1,1},\ldots ,p_{1,j_1},p_{2,1},\ldots, p_{2,j_2},\ldots, p_{\mathcal{I},j_{\mathcal{I}}})$ and $\Tilde{q}=(q_{1,1},\ldots , q_{\mathcal{I},l_{\mathcal{I}}})$ are vectors of primes where the components satisfy $q^{\beta_{i-1}}<p_{i,1},\ldots, p_{i,j_i},q_{i,1},\ldots ,q_{i,l_i} \leq q^{\beta_{i}}$ for each $1\leq i\leq \mathcal{I}$. Moreover
$$C(\Tilde{p},\Tilde{q})=\prod_{1\leq i\leq \mathcal{I}}\prod_{\substack{1\leq r\leq j_i \\ 1\leq s\leq l_i}}\frac{1}{p_{i,r}^{1/2+1/\beta_{\mathcal{I}}\log q}}\frac{\log(q^{\beta_{\mathcal{I}} }/p_{i,r})}{\log(q^{\beta_{\mathcal{I}}}) }\frac{1}{q_{i,s}^{1/2+1/\beta_{\mathcal{I}}\log q}}\frac{\log(q^{\beta_{\mathcal{I}} }/q_{i,s})}{\log(q^{\beta_{\mathcal{I}}}) }.$$
When $\Tilde{j}, \Tilde{l},\Tilde{p},\Tilde{q}$ are fixed, let us denote
$$P=\prod_{1\leq i\leq \mathcal{I}}\prod_{\substack{1\leq r\leq j_i \\ 1\leq s\leq l_i}}p_{i,r}q_{i,s}.$$

Using the identity $\Re z=(z+\bar z)/2$ and the orthogonality of characters the innermost sum in (\ref{bigsum}) (i.e. the sum $\sum_{\chi}\ldots $)  becomes
\begin{align*}
   & \sum_{\chi \in X_q  } \prod_{1\leq i \leq \mathcal{I}}\prod_{\substack{1\leq r\leq j_i \\ 1\leq s\leq l_i}}\frac{\chi(p_{i,r}) h(p_{i,r} )+\bar{\chi}(p_{i,r}) \bar{h}(p_{i,r} )}{2}\cdot \frac{\chi(q_{i,s}) h(q_{i,s} )+\bar{\chi}(q_{i,s}) \bar{h}(q_{i,s} )}{2} \\
= & \phi(q)\bigg(\prod_{1\leq i\leq \mathcal{I}}\frac{1}{2^{j_i+l_i}}\bigg)   \sum_{\Tilde{\delta}, \Tilde{\epsilon}} \prod_{1\leq i\leq \mathcal{I}} \prod_{\substack{r\leq j_i \\ s\leq l_i}} h^{(\delta_{i,r})} (p_{i,r}) h^{(\epsilon_{i,s})} (q_{i,s})\mathbf{1}\bigg((P,q)=1 \, \text{and} \, \prod_{1\leq i\leq \mathcal{I}} \prod_{\substack{r\leq j_i \\ s\leq l_i}} p_{i,r}^{\delta_{i,r}}q_{i,s}^{\epsilon_{i,s}}\equiv 1 \mod q\bigg). \\
\end{align*}
Here $\Tilde{\delta}=(\delta_{1,1},\ldots, \delta_{1,j_1},\ldots, \delta_{\mathcal{I},j_{\mathcal{I}}})$ and $\Tilde{\epsilon}=(\epsilon_{1,1},\ldots, \epsilon_{1,l_1},\ldots, \epsilon_{\mathcal{I},l_{\mathcal{I}}})$ where each component is $-1$ or $+1$. Moreover $h^{(1)}(p)=h(p)$ and $h^{(-1)}(p)=\bar{h}(p)$. Notice that
$$P=\prod_{1\leq i\leq \mathcal{I}} \prod_{\substack{r\leq j_i \\ s\leq l_i}} p_{i,r}q_{i,s}\leq \prod_{1\leq i\leq \mathcal{I}} q^{200\beta_{i}^{1/4}}\ll q^{400\beta_{\mathcal{I}}^{1/4}}\leq q^{0.1},$$
which means that if $(P,q)=1$ then $\prod_{1\leq i\leq \mathcal{I}} \prod_{\substack{r\leq j_i \\ s\leq l_i}} p_{i,r}^{\delta_{i,r}}q_{i,s}^{\epsilon_{i,s}}\equiv 1 \mod q$ is equivalent to $\prod_{1\leq i\leq \mathcal{I}} \prod_{\substack{r\leq j_i \\ s\leq l_i}} p_{i,r}^{\delta_{i,r}}q_{i,s}^{\epsilon_{i,s}}=1$. Hence we can rewrite the above sum as
$$\phi(q)\bigg(\prod_{1\leq i\leq \mathcal{I}}\frac{1}{2^{(j_i+l_i)}}\bigg)   \sum_{\Tilde{\delta}, \Tilde{\epsilon}} \prod_{1\leq i\leq \mathcal{I}} \prod_{\substack{r\leq j_i \\ s\leq l_i}} h^{(\delta_{i,r})} (p_{i,r}) h^{(\epsilon_{i,s})} (q_{i,s})\mathbf{1}\bigg((P,q)=1 \, \text{and} \, \prod_{1\leq i\leq \mathcal{I}} \prod_{\substack{r\leq j_i \\ s\leq l_i}} p_{i,r}^{\delta_{i,r}}q_{i,s}^{\epsilon_{i,s}}=1 \bigg)$$
Notice that the primes $p_{i,r}, q_{i,s}$ are in different range for different $i$, so
$$\mathbf{1}\bigg( \prod_{1\leq i\leq \mathcal{I}} \prod_{\substack{r\leq j_i \\ s\leq l_i}} p_{i,r}^{\delta_{i,r}}q_{i,s}^{\epsilon_{i,s}}=1 \bigg)=\prod_{1\leq i\leq \mathcal{I}} \mathbf{1}\bigg( \prod_{\substack{r\leq j_i \\ s\leq l_i}} p_{i,r}^{\delta_{i,r}}q_{i,s}^{\epsilon_{i,s}}=1\bigg).$$
Therefore we find that (\ref{bigsum}) is
\begin{equation}
\label{bigequation}
\begin{split}
    \leq & \phi(q)\prod_{1\leq i\leq \mathcal{I}} \sum_{0\leq j,l\leq 100\beta_i^{-3/4 } } \frac{1}{j!}\frac{1}{l!}\frac{1}{2^{j+l}} \sum_{\substack{q^{\beta_{i-1}}<p_1,\ldots,p_j,q_1,\ldots, q_l\leq q^{\beta_i}\\ \delta_1,\ldots, \delta_j,\epsilon_1,\ldots, \epsilon_l\in \{-1,1\} }} \mathbf{1}\big(p_1^{\delta_1}\cdots q_l^{\epsilon_l}=1\big) \frac{|h^{(\delta_1)}(p_1)\cdots h^{(\epsilon_l)}(q_l)|}{(p_1\cdots q_l)^{1/2} }\\
    = & \phi(q)\prod_{1\leq i\leq \mathcal{I}} \sum_{m\leq 200\beta_i^{-3/4} } \bigg(\sum_{\substack{j+l=2m\\0\leq j,l\leq 100\beta_i^{-3/4}  }}\frac{1}{j!}\frac{1}{l!}\bigg)\frac{1}{4^m}\cdot\\
    & \cdot \sum_{q^{\beta_{i-1}}<p_1,\ldots, p_m\leq q^{\beta_i}}\frac{|h(p_1)\cdots h(p_m)|^2}{p_1\cdots p_m}\frac{\# \{(q_1,\ldots, q_{2m},\delta_1,\ldots, \delta_{2m}):q_1\cdots q_{2m}=p_1^2\cdots p_m^2 \text{ and } q_1^{\delta_1} \cdots q_{2m}^{\delta_{2m}}=1 \} }{\# \{ (q_1,\ldots, q_m) :q_1\cdots q_m=p_1\cdots p_m  \} }, \\
\end{split}
\end{equation}
where in the last line we substituted $2m=j+l$ where $m$ is a non-negative integer. This can be done since if $j+l$ is odd then $\mathbf{1}(p_1^{\delta_1}\cdots q_k^{\epsilon_l} =1)=0$.  We also introduced a correction factor which ensures that when we changed variables in the summation we count everything the right number of times.
Now assume that in the prime factorisation of $p_1\cdots p_m$ the exponents are $\alpha_1,\alpha_2,\ldots, \alpha_r$ (so $\alpha_1+\cdots +\alpha_r=m$). Then
\begin{align*}
\# \{(q_1,\ldots, q_{2m},\delta_1,\ldots, \delta_{2m}):q_1\cdots q_{2m}=p_1^2\cdots p_m^2 \text{ and } q_1^{\delta_1} \cdots q_{2m}^{\delta_{2m}}=1 \} &= \frac{(2m)!}{\prod_{i=1}^r(2\alpha_i)!}\prod_{i=1}^r \binom{2\alpha_i}{\alpha_i}, \\
\# \{ (q_1,\ldots, q_m) :q_1\cdots q_m =p_1\cdots p_m  \} &=\frac{m!}{\prod_{i=1}^r \alpha_i! }. \\
\end{align*}
Proceeding with this notation and noting that
\begin{equation*}
    \frac{1}{4^m}\sum_{\substack{j+l=2m\\0\leq j,l\leq 100\beta_i^{-3/4}  }}\frac{1}{j!}\frac{1}{l!}\leq \frac{1}{(2m)!}\cdot \frac{1}{4^m}\sum_{0\leq j\leq 2m}\binom{2m}{j}=\frac{1}{(2m)!},
\end{equation*}
we get that (\ref{bigequation}) is 
\begin{align*}
\leq & \phi(q)\prod_{1\leq i\leq \mathcal{I}} \sum_{m\leq 200\beta_i^{-3/4}} \frac{1}{(2m)!}\sum_{q^{\beta_{i-1}}<p_1,\ldots, p_m\leq q^{\beta_i}} \frac{|h(p_1)\cdots h(p_m)|^2}{p_1\cdots p_m} \cdot \frac{ \frac{(2m)!}{\prod_{i=1}^r(2\alpha_i)!}\prod_{i=1}^r \binom{2\alpha_i}{\alpha_i} }{\frac{m!}{\prod_{i=1}^r \alpha_i! }} \\
\leq  & \phi(q)\prod_{1\leq i\leq \mathcal{I}} \sum_{m\leq 200\beta_i^{-3/4}} \frac{1}{m!}\sum_{q^{\beta_{i-1}}<p_1,\ldots, p_m\leq q^{\beta_i}} \frac{|h(p_1)\cdots h(p_m)|^2}{p_1\cdots p_m}\frac{1}{\prod_{i=1}^r \alpha_i!}\\
\leq &\phi(q)\prod_{1\leq i\leq \mathcal{I}} \sum_{m\leq 200\beta_i^{-3/4} } \frac{1}{m!} \bigg(\sum_{q^{\beta_i-1} <p\leq q^{\beta_i}} \frac{|h(p)^2|}{p}\bigg)^m\leq  \phi(q)\exp\bigg(\sum_{p\leq q^{\beta_{\mathcal{I}}}} \frac{|h(p)|^2}{p}\bigg).\\
\end{align*}
We have
$$|h(p)|^2=\sum_{i=1}^{2k}\frac{a_i^2}{4}+\sum_{1\leq i<j\leq 2k} \frac{a_ia_j}{2}\cos(|t_i-t_j|\log p),$$
so Theorem \ref{t1} follows from Lemma \ref{mertenstype}.

\end{proof}

\begin{proof}[Proof of Lemma \ref{mainsecond}]
Let $0\leq j\leq \mathcal{I}-1$ and recall the defintion of $\mathcal{S}(j)$. For each $j+1\leq l\leq \mathcal{I}$ let
$$\mathcal{S}(j,l)=\{\chi \in X_q^*:|\Re G_{(i,l')}(\chi)|\leq \beta_{i}^{-3/4} \; \forall 1\leq i\leq j, \, \forall i\leq l'\leq \mathcal{I}\, \text{but}\, |\Re G_{(j+1,l)}(\chi)|>\beta_{j+1}^{-3/4}  \},$$
so that
$$\mathcal{S}(j)=\bigcup_{l=j+1}^{\mathcal{I}} \mathcal{S}(j,l).$$
By the definition of $\mathcal{S}(j,l)$, for any $M$ positive integer we have
\begin{equation}
\label{b}
\begin{split}
 &\sum_{\chi \in \mathcal{S}(j,l)}\exp^2\bigg( \Re \sum_{p\leq q^{\beta_j}} \frac{\chi(p)h(p)}{p^{1/2+1/(\beta_j\log q)}}\frac{\log(q^{\beta_j}/p)}{\log(q^{\beta_j}) } \bigg)\\
\leq & \sum_{\chi \in \mathcal{S}(j,l)}\exp^2\bigg( \Re \sum_{p\leq q^{\beta_j}} \frac{\chi(p)h(p)}{p^{1/2+1/(\beta_j\log q)}}\frac{\log(q^{\beta_j}/p)}{\log(q^{\beta_j}) } \bigg)\bigg(\beta_{j+1}^{3/4}\Re G_{(j+1,l)}(\chi)\bigg)^{2M}\\
\ll & \beta_{j+1}^{3M/2}\sum_{\chi \in X_q}\prod_{i=1}^j\bigg(\sum_{0\leq n\leq 100k\beta_i^{-3/4}} \frac{(\Re G_{(i,j)}(\chi))^n}{n!}\bigg)^2(\Re G_{(j+1,l)}(\chi))^{2M}. 
\end{split}
\end{equation}
As in the proof of Lemma \ref{mainfirst}, we expand this to a Dirichlet polynomial and switch the order of summation to get cancellation. We would get essentially the same result if it were not for the $(\Re G_{(j+1,l)}(\chi))^{2M}$ term. In order to proceed the same way we need to make sure the the length of the expanded Dirichlet polynomial is less than $q$. This is certainly the case if $(q^{\beta_{j+1}})^{2M}\leq q^{0.2}$. Moreover $\Re G_{(j+1,l)}(\chi)$ runs through primes bigger than $q^{\beta_j}$, so these primes are distinct from the ones occurring in $\Re G_{(i,l)}(\chi)$ when $1\leq i\leq j$. So the extra contribution coming from $\Re G_{(j+1,l)}(\chi)$ is (in the first sum the numbers $\alpha_1,\ldots ,\alpha_r$ will denote the exponents of the prime factorisation of $p_1\cdots p_M$)
\begin{align*}
    \leq & \frac{1}{2^{2M}}  \sum_{q^{\beta_j} <p_1,\ldots ,p_M\leq q^{\beta_{j+1} } }\frac{|h(p_1)\cdots h(p_M)|^2}{p_1\cdots p_M} \frac{ \frac{(2M)!}{\prod_{i=1}^r(2\alpha_i)!}\prod_{i=1}^r \binom{2\alpha_i}{\alpha_i} }{\frac{M!}{\prod_{i=1}^r \alpha_i! }} \\
    \leq & \frac{(2M)!}{M!\cdot 2^{2M}  }\sum_{q^{\beta_j} <p_1,\ldots ,p_M\leq q^{\beta_{j+1} } }\frac{|h(p_1)\cdots h(p_M)|^2}{p_1\cdots p_M} \frac{1}{\prod_{i=1}^r  \alpha_i! } \\
    \leq &  \frac{(2M)!}{M!\cdot 2^{2M}  } \bigg(\sum_{q^{\beta_j }<p\leq q^{\beta_{j+1} } } \frac{|h(p)|^2}{p}\bigg)^M  \\
\end{align*}
Hence, if $(q^{\beta_{j+1}})^{2M}\leq q^{0.2}$, then (\ref{b}) is 
\begin{equation}
\label{biatch}
    \ll \phi(q) \beta_{j+1}^{3M/2} \exp\bigg(\sum_{p\leq q^{\beta_j}}\frac{|h(p)|^2}{p}\bigg) \frac{(2M)!}{M!\cdot  2^{2M}}\bigg(\sum_{q^{\beta_j}< p\leq q^{\beta_{j+1}}} \frac{|h(p)|^2}{p}\bigg)^M.
    \end{equation}
Let us choose $M:=\lfloor 1/(10\beta_{j+1}) \rfloor$. As in the previous lemma we have
$$\exp\bigg(\sum_{p\leq q^{\beta_j}}\frac{|h(p)|^2}{p}\bigg)\ll (\log q)^{(a_1^2+\ldots+a_{2k}^2)/4}\prod_{1\leq i<j\leq 2k}g(|t_i-t_j|)^{a_ia_j/2}=:B$$
Also
$$\bigg(\sum_{q^{\beta_j}< p\leq q^{\beta_{j+1}}} \frac{|h(p)|^2}{p}\bigg)^M<C^M,$$
where $C=10a^2$, where recall that $a=a_1+\cdots a_{2k}+10$. We have $M\geq C^{10}$, so
\begin{align*}
  \beta_{j+1}^{3M/2}\frac{(2M)!}{M!\cdot  2^{2M}}C^M 
 \leq & \exp\Big(\frac{3M}{2}\log \beta_{j+1}+M\log M +M\log C\Big) \\
 \leq & \exp\Big(-\frac{2}{5}M\log M \Big) \\
 \leq &\exp\Big(-\frac{1}{100\beta_{j+1} } \log(\beta_{j+1}^{-1}) \Big). \\
\end{align*}
Hence, for each $1\leq j\leq \mathcal{I}-1$ we have
\begin{align*}
\sum_{\chi \in \mathcal{S}(j)}\exp^2\bigg( \Re \sum_{p\leq q^{\beta_j}} \frac{\chi(p)h(p)}{p^{1/2+1/(\beta_j\log q)}}\frac{\log(q^{\beta_j}/p))}{\log(q^{\beta_j}) }\bigg) 
\leq &  \sum_{j+1\leq l\leq \mathcal{I} }\sum_{\chi \in \mathcal{S}(j,l)} \exp^2\bigg( \Re \sum_{p\leq q^{\beta_j}} \frac{\chi(p)h(p)}{p^{1/2+1/(t_j\log q)}}\frac{\log(q^{\beta_j}/p))}{\log(q^{\beta_j}) }\bigg) \\
\ll & (\mathcal{I}-j)B \exp\Big(-\frac{1}{100\beta_{j+1} } \log(\beta_{j+1}^{-1}) \Big) \\
\ll & \log (1/\beta_j) B   \exp\Big(-\frac{1}{100\beta_{j+1} } \log(\beta_{j+1}^{-1}) \Big)\\
\ll & B \exp\Big(-\frac{1}{200\beta_{j+1} } \log(\beta_{j+1}^{-1}) \Big), \\
\end{align*}
which proves the lemma when $j\geq 1$. When $j=0$, the Dirichlet polynomial $\sum_{p\leq q^{\beta_j}}\ldots$ is empty, so what we infer is that $|\mathcal{S}(0)|\leq B\exp\Big(-\frac{1}{200\beta_{1} } \log(\beta_{1}^{-1}) \Big)\leq q e^{-(\log \log q)^2 }$, so the lemma is proved.
\end{proof}
\begin{proof}[Proof of Lemma \ref{important}]
We show how to modify the proof of Lemma \ref{mainfirst} to take into account the contribution coming from the squares of primes, the modification of Lemma \ref{mainsecond} is similar. The quantity $C$ will always denote a constant, which might depend on the fixed parameters of Theorem \ref{t1}. Moreover the value of $C$ might not be the same at each occurrence. There can be at most one primitive and quadratic $\chi$ mod $q$. The contribution from one such character is negligible, since we have the bound $|L(1/2+it_j,\chi)|\ll \exp\big(C\frac{\log q}{\log \log q}\big)$  when  $t_j=q^{O(1)}$ (see Corollary 2.4 in \cite{munschshifted}). Let $2\leq x\leq q$. If $\chi$ is not a quadratic character then $\chi^2$ is non-trivial, so under GRH a standard explicit formula argument gives us the bound 
$$\sum_{n\leq x} \chi^2(n)\Lambda(n)n^{-it}\ll x^{1/2} \big(\log qx(|t|+2) \big)^2.$$
Recall that by definition $h(p^2)=\frac{1}{2}(a_1p^{-2it_1}+\cdots +a_{2k}p^{-2it}\big)$. Therefore, by partial summation we obtain the bound
$$\sum_{(\log q)^{10A}< p\leq q^{\beta_{ \mathcal{I} } } }\frac{\chi(p^2) h(p^2)}{ p}\ll 1.$$
Moreover by Mertens' second estimate
$$\sum_{\log q< p\leq  (\log q)^{10A} }\frac{\chi(p^2) h(p^2)}{ p}\ll 1,$$
so in fact me may truncate the Dirichlet polynomial coming from the squares of primes at $\log q$ instead of $q^{\beta_{\mathcal{I}}}$. 
For each $1\leq m\leq (\log \log q)/(\log 2)$ let
$$P_m(\chi)=\frac{1}{2} \sum_{2^m\leq p< 2^{m+1} } \frac{\chi(p^2) h(p^2)}{p},$$
and let
$$\mathcal{P}(m)=\{ \chi\in X_q^*: |\Re P_m(\chi)|\geq 2^{-m/10}:\, \text{but} \; |\Re P_n(t)|\leq 2^{-n/10} \text{ } \forall m< n\leq \log\log q/\log 2 \}.$$
If $\chi$ is not inside $\cup_{1\leq m\leq \log\log q/2} \mathcal{P}(m) $, then $\sum_{p\leq \log q} \frac{\chi(p^2)h(p^2)}{p}\ll 1$, which contributes negligibly to the final expression, so from now on we assume $\chi\in \cup_{1\leq m\leq \log\log q/2} \mathcal{P}(m) $,
Let $M=M(m):=\lfloor 2^{3m/4}  \rfloor$. Then, as in the proof of Lemma \ref{mainsecond} where we considered the contribution from $\Re G_{j+1,l}(\chi)$, we get 
$$\# \mathcal{P}(m)\leq \sum_{\chi \in X_q}\big|2^{m/10}\Re P_m(t)|^{2M}\ll 2^{mM/5}\phi(q) \frac{(2M)!}{M!2^{2M}}\bigg( \sum_{2^m<p\leq 2^{m+1}} \frac{|h(p^2)|^2}{p^2}\bigg)^M\leq \phi(q) \bigg( \frac{CM}{2^{4m/5} }\bigg)^M\ll q e^{-2^{3m/4}} , $$
hence the contribution coming from $\chi \in \mathcal{P}(m)$, where $m\geq \frac{3}{2\log 2} \log\log\log q$ is negligible, so from now on we assume $m\leq \frac{3}{2\log 2} \log\log\log q$.
By the trivial bound we have
$$ \bigg| \Re \sum_{p\leq 2^{m+1}} \frac{\chi(p)h(p)}{p^{1/2+1/(\beta_{\mathcal{I}}\log q)}}\frac{\log(q^{\beta_{\mathcal{I} } }/p)) }{\log(q^{\beta_{\mathcal{I} } })}+\frac{1}{2}\Re\sum_{p\leq \log q}\frac{\chi(p^2)h(p^2)}{p}\bigg|\leq \sum_{p\leq 2^{m+1}} \frac{1}{p^{1/2} } +\frac{1}{2}\sum_{p\leq 2^{m+1}}\frac{1}{p}+O(1) \leq C 2^{m/2}.$$

Note that if $1\leq m\leq \frac{3}{2\log 2} \log\log\log q$ , then $2^{m/2}=o(\beta_0^{-3/4})$ (recall that $\beta_0=\frac{1}{(\log \log q)^2}$), thus for any $\chi\in \mathcal{T}$ we have
\begin{equation*}
    \bigg|\Re \sum_{2^{m+1}< p\leq q^{\beta_{0 }}} \frac{\chi(p)h(p)}{p^{1/2+1/(\beta_{\mathcal{I}}\log q)}}\frac{\log(q^{\beta_{\mathcal{I}} }/p))}{\log(q^{\beta_{\mathcal{I}}}) }\bigg|\leq (1+o(1))\beta_0^{-3/4},
\end{equation*}
so we can run the exact same argument as in Lemma \ref{mainfirst} on this slightly truncated polynomial. Therefore we get
\begin{align*}
   & \sum_{\chi\in  \mathcal{P}(m) \cap \mathcal{T} }\exp^2\bigg( \Re \sum_{p\leq q^{\beta_{\mathcal{I} } }} \frac{\chi(p)h(p)}{p^{1/2+1/(\beta_{\mathcal{I}}\log q)}}\frac{\log(q^{\beta_{\mathcal{I}}}/p)}{\log(q^{\beta_{\mathcal{I}}}) }+\frac{1}{2}\Re \sum_{p\leq \log q} \frac{\chi(p^2)h(p^2) }{p}\bigg) \\
\leq & e^{C2^{m/2}} \sum_{\chi\in  \mathcal{P}(m) \cap \mathcal{T} } \exp^2\bigg( \Re \sum_{2^{m+1}< p\leq q^{\beta_{\mathcal{I} }}} \frac{\chi(p)h(p)}{p^{1/2+1/(\beta_{\mathcal{I}}\log q)}}\frac{\log(q^{\beta_{\mathcal{I}} }/p)}{\log(q^{\beta_{\mathcal{I}}}) }\bigg)\\
\ll & e^{C2^{m/2} } \sum_{\chi\in  \mathcal{P}(m) \cap \mathcal{T} }(2^{m/10}\Re P_m(\chi))^{2M} \exp^2\bigg( \Re \sum_{2^{m+1}< p\leq q^{\beta_{\mathcal{I} }}} \frac{\chi(p)h(p)}{p^{1/2+1/(\beta_{\mathcal{I}}\log q)}}\frac{\log(q^{\beta_{\mathcal{I}} }/p)}{\log(q^{\beta_{\mathcal{I}}}) }\bigg) \\
\ll & e^{C2^{m/2}} 2^{mM/5} \phi(q)\cdot  \frac{(2M)!}{2^{2M} M!}  \bigg(\sum_{2^m<p\leq 2^{m+1}} \frac{|h(p^2)|^2}{p^2}\bigg)^M \exp\bigg(\sum_{2^{m+1}<p\leq q^{\beta_{\mathcal{I}}} }\frac{|h(p^2)|^2}{p} \bigg) \\
 \ll  & e^{C2^{m/2}-2^{3m/4}} B. \\
\end{align*}
Summing this over $1\leq m\leq \frac{3}{2\log 2} \log\log\log q$ we get the desired bound.
\end{proof}

\section{Proof of Theorem \ref{t2}}
We prove the theorem for the even case, the odd case is essentially identical. Recall, we have
$$S_{2k}^{+}(q):=\sum_{\chi \in X_q^+}|\theta(1,\chi)|^{2k}.$$
We may write $\theta(1,\chi)$ as the inverse Mellin transform of $L(s,\chi)$ twisted by some additional factor. For an even character $\chi$ and $c>1/2$ one has
$$\theta(1,\chi)=\frac{1}{2\pi i} \int_{(c)} L(2s, \chi) \bigg(\frac{q}{\pi}\bigg)^{s} \Gamma(s)ds.$$
% $$\mathcal{M}(\theta(x,\chi) , s)=\bigg(\frac{q}{\pi}\bigg)^s \Gamma(s)L(2s,\chi)$$
% $$\theta(x,\chi) =\frac{1}{2\pi i}\int_{(c)}L(2s,\chi)\Gamma(s)\bigg(\frac{q}{\pi}\bigg)^s x^{-s} ds$$
We let $f(t, \chi)=L(1/2+it, \chi)\Gamma(1/4+it/2)(q/\pi)^{it/2}$. Shifting the line of integration to $c=1/4$ and noting that we sum positive quantities so we can sum over $X_q^*$ instead of $X_q^+$, we get
%s=1/4+it/2
\begin{align*}
     S_{2k}^{+}(q) & \ll   \sum_{\chi \in X_q^+} \bigg| \int_{-\infty}^{\infty} f(t,\chi)\bigg( \frac{q}{\pi}\bigg)^{1/4} dt \bigg|^{2k} \\ 
    & \ll \sum_{\chi\in X_q^*} \bigg| \int_{-\infty}^{\infty} f(t,\chi)\bigg( \frac{q}{\pi}\bigg)^{1/4} \mathbf{1}(|t|\leq q) dt \bigg|^{2k}
    +\sum_{\chi\in X_q^*} \bigg| \int_{-\infty}^{\infty} f(t,\chi)\bigg( \frac{q}{\pi}\bigg)^{1/4} \mathbf{1}(|t|> q) dt \bigg|^{2k}=:I_1+I_2.\\
\end{align*}
$I_2$ can be easily bounded via the exponential decay of $\Gamma(1/4+it/2)$. We have $|\Gamma(1/4+it/2)|\ll e^{-t/10}$ for all $t\in \mathbb{R}$ and also $|L(1/2+it,\chi)|\ll q(|t|+2)$ for any $\chi\in X_q^*$ (see Theorem C.1. and Lemma 10.15 of \cite{montgomery2007multiplicative}), hence the crude bound
$$I_2\ll q\cdot \bigg(\int_q^{\infty} (qt)^2 e^{-t/10} dt \bigg)^{2k}\ll 1.$$
In the rest of the proof we bound $I_1$. In fact, we start by reducing the problem as follows. 
\begin{lemma}
\label{seged}
    Assume that for each $0\leq n\leq q$ integer we have the moment inequality
    \begin{equation*}
        \sum_{\chi \in X_q^*} \bigg(\int_{n}^{n+1} |L(1/2+it, \chi)| dt\bigg)^{2k}\ll \phi(q)(\log q)^{(k-1)^2}.
    \end{equation*}
Then Theorem \ref{t2} holds.
\end{lemma}
\begin{proof}
    Note that $I_1$ involves integrals ranging from $-q$ to $q$.
We would like to break up these integrals into smaller pieces of length one, which we can do using Hölder's inequality. We note that $|f(-t,\bar\chi)|=|f(t,\chi)|$, and $\bar\chi$ is even if $\chi$ is, so putting absolute values inside the integral we get

\begin{align*}
    I_1 & \ll q^{k/2} \sum_{\chi \in X_q^*} \bigg( \sum_{0\leq n\leq q}\frac{1}{n+1}\cdot  (n+1) \int_{n}^{n+1} |f(t,\chi)|dt\bigg)^{2k} \\
    & \ll q^{k/2} \sum_{\chi \in X_q^*} \bigg(\sum_{0\leq n\leq q}(n+1)^{2k}\bigg(\int_{n}^{n+1} |f(t,\chi)|dt\bigg)^{2k} dt \bigg)\cdot \bigg(\sum_{0\leq n\leq q} (n+1)^{-2k/(2k-1)}\bigg)^{2k-1} \\
    &  \ll q^{k/2} \sum_{0\leq n\leq q}(n+1)^{2k}  e^{-n/10} \sum_{\chi \in X_q^*} \bigg(\int_{n}^{n+1} |L(1/2+it, \chi)| dt\bigg)^{2k}  , \\
\end{align*}
where in the last line we used $|\Gamma(1/4+it/2)|\ll e^{-t/10}$. Here we introduced the auxiliary weights $\frac{1}{n+1}$, so we do not lose more than a constant when applying Hölder (see the rightmost sum in the second line). The price we pay is that the term $(n+1)^{2k}$ appears in the final expression, however this is offset by the exponential decay of the $\Gamma$ function, which is responsible for the term $e^{-n/10}$. We note that the exact choice of these weights is not important.

If the assumption of the lemma holds, noting that $\sum_{0\leq n\leq q}(n+1)^{2k}  e^{-n/10} \ll 1$, we obtain
 $$I_1\ll q^{k/2}\phi(q)(\log q)^{(k-1)^2},$$
 which implies Theorem \ref{t2}.
\end{proof}
Let $0\leq n\leq q$ be an integer. In the rest of the proof we show that indeed
\begin{equation}
\label{momentone}
    \sum_{\chi \in X_q^*} \bigg(\int_{n}^{n+1} |L(1/2+it, \chi)| dt\bigg)^{2k}\ll \phi(q)(\log q)^{(k-1)^2}.
\end{equation}
% If $2k$ is an integer the left hand side is in fact
% $$\int_{[n,n+1]^{2k}}\sum_{\chi \in X_q^*} \prod_{j=1}^{2k}|L(1/2+it_j,\chi)| d\mathbf{t},$$
% which we can bound using Theorem 1 since $n\leq q$.
 We first pull out 3 copies of the integral as outlined previously, which we can do since $2k>4$ by assumption. Using the notation $d\mathbf{t}=dt_1dt_2dt_3$, we obtain
\begin{align*}
    \bigg(\int_{n}^{n+1} |L(1/2+it, \chi)| dt\bigg)^{2k} &\leq \int_{[n,n+1]^3}\prod_{a=1}^3|L(1/2+it_a, \chi)| \bigg(\int_{[n,n+1]}|L(1/2+iu, \chi)|du \bigg)^{2k-3} d\mathbf{t}\\
     & \ll \int_{[n,n+1]^3}\prod_{a=1}^3|L(1/2+it_a, \chi)| \bigg(\int_{\mathcal{D} }|L(1/2+iu, \chi)|du \bigg)^{2k-3} d\mathbf{t}, \\
\end{align*}
where $\mathcal{D}=\mathcal{D}(t_1,t_2,t_3)=\{ u\in [n,n+1]:|t_1-u|\leq |t_2-u|\leq |t_3-u| \}$. Here, we may restrict the integration over $\mathcal{D}$ by symmetry. In fact, later on in the argument it will be important that we have a fixed ordering of the distances $|t_a-u|$.

We would like to apply Hölder to bring the $(2k-3)$-th power inside the integrand. However, doing that directly would be too wasteful, so we first need to partition the region $\mathcal{D}$ according to the distance $|t_1-u|$. This ensures that we extract the main contribution where variables are close together.

We partition $[-1,1]$ into dyadic regions as follows. Let $\mathcal{B}_1=\big[-\frac{1}{\log q},\frac{1}{\log q}\big]$. For $2\leq j< \lfloor\log\log q\rfloor+1=:K$ let $\mathcal{B}_j=\big[-\frac{e^{j-1}}{\log q}, -\frac{e^{j-2}}{\log q}\big]\cup \big[\frac{e^{j-2}}{\log q}, \frac{e^{j-1}}{\log q}\big] $. Finally, we define $\mathcal{B}_K=[-1,1]\setminus \bigcup_{1\leq j<K} \mathcal{B}_j $. 

For any $t_1\in [n,n+1]$ we have 
$$\mathcal{D}\subset [n,n+1] \subset t_1+[-1,1]\subset \bigcup_{1\leq j\leq K} t_1+\mathcal{B}_j,$$ 
so if we let $\mathcal{A}_j=\mathcal{B}_j\cap (-t_1+\mathcal{D})$ then $(t_1+\mathcal{A}_j)_{1\leq j\leq K}$ form a partition of $\mathcal{D}$. Applying Hölder twice we get
\begin{align*}
     \bigg(\int_{\mathcal{D}}|L(1/2+iu, \chi)|du\bigg)^{2k-3}& \leq \bigg( \sum_{1\leq j\leq K} \frac{1}{j}\cdot  j \int_{t_1+\mathcal{A}_j} |L(1/2+iu, \chi)|du  \bigg)^{2k-3} \\
    & \leq \bigg(\sum_{1\leq j\leq K} j^{2k-3} \bigg( \int_{t_1+\mathcal{A}_j} |L(1/2+iu,\chi)|du  \bigg)^{2k-3}\bigg) \bigg(\sum_{1\leq j\leq K} j^{-\cdot (2k-3)/(2k-4)} \bigg)^{2k-4} \\
    & \ll \sum_{1\leq j\leq K} j^{2k-3} \bigg( \int_{t_1+\mathcal{A}_j} |L(1/2+iu, \chi)| \bigg)^{2k-3} \\
    & \leq \sum_{1\leq j\leq K} j^{2k-3} |\mathcal{B}_j|^{2k-4} \int_{t_1+\mathcal{A}_j} |L(1/2+iu, \chi)|^{2k-3}du, \\
\end{align*}
where in the last line we used that $|\mathcal{A}_j|\leq |\mathcal{B}_j| $. Here the introduction of the weights $\frac{1}{j}$ has similar purpose to the weights $\frac{1}{n+1}$ in the proof of Lemma \ref{seged}.
For simplicity, for $\mathbf{t}=(t_1,t_2,t_3)$ let us denote
$$L(\mathbf{t},u)= \sum_{\chi\in X_q^*}|L(1/2+iu, \chi)|^{2k-3}\prod_{a=1}^3|L(1/2+it_a, \chi)|  .$$
Hence we have shown that
\begin{equation}
\label{region}
\sum_{\chi \in X_q^*} \bigg(\int_{n}^{n+1} |L(1/2+it, \chi)| dt\bigg)^{2k}\ll \sum_{1\leq j\leq K}j^{2k-3}|\mathcal{B}_j|^{2k-4}\int_{[n,n+1]^3} \int_{t_1+\mathcal{A}_j} L(\mathbf{t},u)dud\mathbf{t}.
\end{equation}
We are now going to partition the integral into smaller regions, where each region is such that $L(\mathbf{t},u)$ does not change by more than a fixed constant, Therefore, when applying Theorem \ref{t1} we do not lose more than a constant.

Note that we have already restricted our integral to a region where $|t_1-u|$ is fixed up to a constant and $|t_1-u|\leq|t_2-u|\leq |t_3-u|$. We partition this region into smaller regions where $|t_2-u|-|t_1-u|$ and $|t_3-u|-|t_2-u|$ are also fixed (up to a constant). More precisely, for each $1\leq j,l,m\leq K$, let $$\mathcal{C}_{j,l,m}=\{(t_1,t_2,t_3,u)\in [n,n+1]^4: u\in t_1+ \mathcal{A}_j,\, |t_2-u|-|t_1-u|\in \mathcal{B}_l,\, |t_3-u|-|t_2-u|\in \mathcal{B}_m \}.$$
 Recall that by the definition of $\mathcal{A}_j$, $u\in t_1+ \mathcal{A}_j$ implies $|t_1-u|\leq |t_2-u|\leq |t_3-u|$. Since $|t_2-u|-|t_1-u|$ and $|t_3-u|-|t_2-u|$ are inside $[0,1]$, for a fixed $1\leq j\leq K$ we have $[n,n+1]^3\times (t_1+\mathcal{A}_j)\subset \bigcup_{1\leq l,m\leq K} \mathcal{C}_{j,l,m}$, so this is a partition indeed.
 
Let us consider the volume of $\mathcal{C}_{j,l,m}$. If we fix $u$ then $t_1$ is in a fixed region of size $\ll \frac{e^j}{\log q}$. If we fix $u$ and $t_1$ then $t_2$ is in a fixed region of size $\ll \frac{e^l}{\log q} $ (since $\pm(t_2-u)\in |t_1-u|+\mathcal{B}_l$), if we fix $u$ and $t_2$ then $t_3$ is in a fixed region of size $\ll \frac{e^m}{\log q} $. Hence the volume of $\mathcal{C}_{j,l,m}$ is $\ll \frac{e^{j+l+m}}{(\log q )^3} $.

We know give an upper bound on $L(\mathbf{t},u)$ knowing that $(t_1,t_2,t_3,u)\in \mathcal{C}_{j,l,m}$ using Theorem 1. By the definition of $\mathcal{C}_{j,l,m}$ we have
$\frac{e^j}{\log q}\ll |t_1-u|\ll 1$ so $g(|t_1-u|)\ll \frac{\log q}{e^j}$. By the definition of $\mathcal{A}_j$ we have $|t_2-u|\geq |t_1-u|$, so $|t_2-u|= |t_1-u|+(|t_2-u|-|t_1-u|)\gg \frac{e^j}{\log q}+\frac{e^l}{\log q}$ hence $g(|t_2-u|)\ll \frac{\log q}{e^{\max(j,l) }}$. By similar considerations we obtain  $g(|t_3-u|)\ll \frac{\log q}{e^{\max(j,l,m) }}$, $g(|t_2-t_1|)\ll \frac{\log q}{e^l}$, $g(|t_3-t_2|)\ll \frac{\log q}{e^m}$ and finally $g(|t_3-t_1|) \ll \frac{\log q}{e^{\max(l,m ) } } $.
So by Theorem 1, if $(t_1,t_2,t_3,u)\in \mathcal{C}_{j,l,m}$, then
\begin{align*}
     L(\mathbf{t},u) 
     \ll & \phi(q)(\log q)^{k^2-3k+3} \bigg(\frac{\log q}{e^j}\cdot \frac{\log q}{e^{\max \{j,l \}} }\cdot  \frac{\log q}{e^{ \{j,l,m \} }} \bigg)^{(2k-3)/2}\bigg(\frac{\log q}{e^l}\cdot \frac{\log q}{e^{m}}\cdot \frac{\log q}{e^{\max \{l,m \}}} \bigg)^{1/2} \\
     = & \phi(q)(\log q)^{k^2} \exp\Big( -\frac{2k-3}{2}(j+\max \{j,l \}+\max \{j,l,m \})-\frac{1}{2}(l+m+\max \{l,m \} ) \Big). \\
\end{align*}
Now that we have obtained an upper bound on the volume of $\mathcal{C}_{j,l,m}$ and an upper bound on $L(\mathbf{t},u)$ when $(t_1,t_2,t_3,u)\in \mathcal{C}_{j,l,m}$, recalling that $|\mathcal{B}_j|\leq \frac{e^j}{\log q}$, we may continue \eqref{region} as follows. 
\begin{align*}
        & \sum_{1\leq j\leq K} j^{2k-3} |\mathcal{B}_j|^{2k-4} \int_{[n,n+1]^3}\int_{t_1+\mathcal{A}_j}  L(\mathbf{t},u)dud\mathbf{t} \\
        \ll & \sum_{1\leq j\leq K} j^{2k-3} |\mathcal{B}_j|^{2k-4} \sum_{1\leq l,m\leq K} \text{Volume}(\mathcal{C}_{j,l,m}) \cdot \sup_{(\mathbf{t},u)\in \mathcal{C}_{j,l,m}}  L(\mathbf{t},u) \\
    \ll & \phi(q)(\log q)^{(k-1)^2}\sum_{1\leq j, l, m\leq K} j^{2k-3} \exp\Big( j(k-3/2)+l/2+m/2-(k-3/2)( \max \{j,l \}+\max \{j,l,m \})-\max\{l,m \}/2\Big) \\
    \ll &  \phi(q)(\log q)^{(k-1)^2} \sum_{1\leq j,l,m\leq K} j^{2k-3} e^{-(k-2)\max \{j,l,m \}} \\
    \ll &  \phi(q)(\log q)^{(k-1)^2} \sum_{1\leq j,l,m\leq K} j^{2k-3} e^{-(k-2)(j+l+m)/3} \\
    \ll &  \phi(q)(\log q)^{(k-1)^2},  \\
\end{align*}
which proves Theorem 2. Note that in the last line we used $k>2$.
\section{Proof of theorem 3}
In this section we prove Theorem 3. We break this section up into three subsections. In the first one we state a variant of Theorem \ref{t1}, which concerns shifted moments of $L$-functions off the critical line. The statement can be proved very similarly to Theorem \ref{t1}, we will indicate the changes one has to make. In the second and third subsections we will prove the first and second part of Theorem \ref{t3} respectively.
\subsection{A variant of Theorem \ref{t1}}
\begin{theorem}
\label{t4}
Let $2k\geq 1$ be a fixed integer and $a_1,\ldots, a_{2k}, A$ fixed positive real numbers, $2\leq y\leq q$. Suppose the Dirichlet L-functions modulo $q$ satisfy the Generalised Riemann Hypothesis. Let $X_q^*$ denote the set of primitive characters modulo $q$. Let $t=(t_1,\ldots ,t_{2k})$ be a real $2k$-tuple with $|t_j|\leq  q^A$. Then
$$\sum_{\chi \in X_q^*} \big| L\big(1/2+1/\log y+ it_1,\chi \big) \big|^{a_1} \cdots \big| L\big(1/2+1/\log y+ it_{2k},\chi \big) \big|^{a_{2k}}\ll \phi(q)(\log y)^{(a_1^2+\cdots +a_{2k}^2)/4} \prod_{1\leq i<j\leq 2k} g^*(|t_i-t_j|)^{a_ia_j/2},  $$
where $g^*:\mathbb{R}_{\geq 0} \rightarrow \mathbb{R}$ is the function defined by
$$g^*(x) =\begin{cases}
\log y  & \text{if } x\leq \frac{1}{\log y} \text{ or } x\geq e^y, \\
\frac{1}{x} & \text{if }   \frac{1}{\log y}\leq x\leq 10, \\
\log \log x & \text{if } 10\leq x\leq e^y
\end{cases}.
$$
\end{theorem}
The proof is essentially the same as that of Theorem \ref{t1}, we briefly indicate the changes one needs to make. Firstly, we use Lemma \ref{variant} to upper bound $|L(s,\chi)|$. Recall, if $2\leq x\leq q$, then
\begin{equation*}
    \log |L(1/2+1/\log y+it, \chi)| \leq \Re \sum_{n\leq x}\frac{\chi(n)\Lambda(n)}{n^{1/2+\max(1/\log y, 1/\log x) +it} \log n}\frac{\log x/n}{\log x}+\frac{\log q+\log^+ t}{\log x}+O(1/\log x).
\end{equation*}
% If $y<x$, we use equation 2.8. in \cite{munschshifted}. We substitute $s_0=1/2+1/\log y+it$ and note that the terms involving $F_{\chi}(s_0)$ have negative contribution, since $F_{\chi}(s_0)>0$ and $y<x$. If $y\geq x$, then in Proposition 2.3 of \cite{munschshifted} we substitute $\lambda=1$ and $\sigma=1/2+1/\log y$.

After that, we run the argument the same way as in the proof of Theorem \ref{t1}, the only (and essential) difference is that our Dirichlet polynomials over primes are weighted by $p^{-1/2-1/\log y}$ instead of $p^{-1/2}$. In particular, if we recall the definition of $C(\Tilde{p},\Tilde{q})$ and $P$ from the proof of Lemma \ref{mainfirst}, there we used $C(\Tilde{p},\Tilde{q})\leq P^{-1/2}$. In the proof of Theorem \ref{t4}, this becomes $C(\Tilde{p},\Tilde{q})\leq P^{-1/2-1/\log y}$. This extra saving is responsible for the difference in the upper bounds in Theorem \ref{t1} and Theorem \ref{t4} (the former involves powers of $\log q$, the latter has powers of $\log y$). To see this, following the proof of Lemma \ref{mainfirst} we get the upper bound
\begin{equation*}
\label{shiftedupperbound}
 \sum_{\chi \in \mathcal{T} } \big| L\big(1/2+1/\log y+ it_1,\chi \big) \big|^{a_1} \cdots \big| L\big(1/2+1/\log y+ it_{2k},\chi \big) \big|^{a_{2k}}\ll \phi(q)\exp\bigg(\sum_{p\leq q}\frac{|h(p)|^2}{p^{1+2/\log y} }\bigg),
\end{equation*}
where recall
$$h(p):=\frac{1}{2}(a_1p^{-it_1}+\cdots +a_{2k}p^{-it_{2k}}).$$
We uniformly have $|h(p)|\ll 1$ and that
$$\sum_{p>y} \frac{1}{p^{1+2/\log y} }\ll 1,$$
so the desired upper bound follows by Lemma \ref{mertenstype}.

\subsection{Proof of the first part of Theorem \ref{t3}}
We begin with a proposition that gives an upper bound on the integral of $L$-functions slightly to the right of the critical line. It can be thought of as a generalisation of (\ref{momentone}).
\begin{proposition}
\label{t3prop}
Let $10\leq B=q^{O(1)}$ and $2\leq y\leq q$. Then 
$$\sum_{\chi \in X_q^*}\bigg(\int_{0}^{B}|L(1/2+1/\log y+it,\chi)|dt\bigg)^{2k}\ll \phi(q)\big( B^3(\log \log B)^{O(1)}(\log y)^{(k-1)^2}+B^{2k}(\log \log B\cdot \log\log y)^{O(1)}(\log y)^{k^2-3k+3}\big)$$
% $$\sum_{j,l,m} j^{2(2k-3)}|\mathcal{B}_j|^{2k-4}\int_{}\sum_{\chi}\prod_{i=1}^3|L(1/2+1/\log y+it_i, \chi)|  |L(1/2+1/\log y + iu, \chi)|^{2k-3}$$
\end{proposition}
\begin{proof}
The proof is very similar to that of (\ref{momentone}), so we will not explain it in much detail, but go through the main steps. For each $\chi$, by symmetry
\begin{align*}
    \bigg(\int_{0}^{B} |L(1/2+1/\log y+ it, \chi)| dt\bigg)^{2k} 
      \ll \int_{[0,B]^3}\prod_{a=1}^3|L(1/2+1/\log y+ it_a, \chi)| \bigg(\int_{\mathcal{D} }|L(1/2+1/\log y+iu, \chi)|du \bigg)^{2k-3} d\mathbf{t}, \\
\end{align*}
where $\mathcal{D}=\mathcal{D}(t_1,t_2,t_3)=\{ u\in [0,B]:|t_1-u|\leq |t_2-u|\leq |t_3-u| \}$.

Let $\mathcal{B}_1=\big[-\frac{1}{\log y},\frac{1}{\log y}\big]$. For $2\leq j< \lfloor\log\log y\rfloor+1=:K$ let $\mathcal{B}_j=\big[-\frac{e^{j-1}}{\log y}, -\frac{e^{j-2}}{\log y}\big]\cup \big[\frac{e^{j-2}}{\log y}, \frac{e^{j-1}}{\log y}\big] $. Let $\mathcal{B}_K=[-B,B]\setminus \bigcup_{1\leq j<K} \mathcal{B}_j $.

Then, for any $t_1\in [0,B]$ we have $\mathcal{D}\subset [0,B] \subset t_1+[-B,B]\subset \bigcup_{1\leq j\leq K} t_1+\mathcal{B}_j$, so if we let $\mathcal{A}_j=\mathcal{B}_j\cap (-t_1+\mathcal{D})$ then $(t_1+\mathcal{A}_j)_{1\leq j\leq K}$ form a partition of $\mathcal{D}$. Applying Hölder twice we get
\begin{align*}
    & \bigg(\int_{\mathcal{D}}|L(1/2+1/\log y+ iu, \chi)|du\bigg)^{2k-3} \\
      \leq & \bigg( \sum_{1\leq j\leq K} \frac{1}{j}\cdot  j \int_{t_1+\mathcal{A}_j} |L(1/2+1/\log y+iu, \chi)|du  \bigg)^{2k-3} \\
     \leq & \bigg(\sum_{1\leq j\leq K} j^{2k-3} \bigg( \int_{t_1+\mathcal{A}_j} \big|L(1/2+1/\log y+iu, \chi)\big|du  \bigg)^{2k-3}\bigg) \bigg(\sum_{1\leq j\leq K } j^{ (2k-3)/(2k-4)} \bigg)^{2k-4} \\
     \ll & \sum_{1\leq j\leq K} j^{2k-3} \bigg( \int_{t_1+\mathcal{A}_j} |L(1/2+1/\log y+iu, \chi)| \bigg)^{2k-3} \\
     \leq & \sum_{1\leq j\leq K} j^{2k-3} |\mathcal{B}_j|^{2k-4} \int_{t_1+\mathcal{A}_j} |L(1/2+1/\log y+iu, \chi)|^{2k-3}du. \\
\end{align*}
For simplicity, for $\mathbf{t}=(t_1,t_2,t_3)$ let us denote
$$L(\mathbf{t},u)= \sum_{\chi\in X_q^*}\prod_{a=1}^3|L(1/2+1/\log y+it_a, \chi)|  |L(1/2+1/\log y+iu, \chi)|^{2k-3}.$$
We have shown that
\begin{align*}
    \sum_{\chi \in X_q^*}\bigg(\int_{0}^{B}|L(1/2+1/\log y+it,\chi)|dt\bigg)^{2k}\ll & \sum_{1\leq j\leq K} j^{2k-3} |\mathcal{B}_j|^{2k-4} \int_{[0,B]^3}\int_{t_1+\mathcal{A}_j} L(\mathbf{t},u)du d\mathbf{t}  \\
     \ll &  \sum_{1\leq j, l, m\leq K} j^{2k-3} |\mathcal{B}_j|^{2k-4} \int_{\mathcal{C}_{j,l,m}} L(\mathbf{t},u)du d\mathbf{t}, 
\end{align*}
where
$$\mathcal{C}_{j,l,m}=\{(t_1,t_2,t_3,u)\in [0,B]^4: u\in t_1+ \mathcal{A}_j,\, |t_2-u|-|t_1-u|\in \mathcal{B}_l,\, |t_3-u|-|t_2-u|\in \mathcal{B}_m \}.$$
We now separate two cases in the summation according to the size of $j$.

\textbf{Case 1:} $j<K$. This is essentially the same as the proof of \eqref{momentone}. The volume of the region $\mathcal{C}_{j,l,m}$ is $V(\mathcal{C}_{j,l,m}) \ll  B^3\frac{e^{j+l+m} }{(\log y)^3}$. Moreover, when $(\mathbf{t},u)\in \mathcal{C}_{j,l,m} $
$$L(\mathbf{t},u)\ll \phi(q)(\log y)^{k^2}(\log \log B)^{O(1)} \exp\Big( -\frac{2k-3}{2}(j+\max \{j,l \}+\max \{j,l,m \})-\frac{1}{2}(l+m+\max \{l,m \} ) \Big).$$
% \begin{align*}
%     & \prod_{i=1}^3|L(1/2+1/\log y+it_i, \chi)|  |L(1/2+1/\log y + iu, \chi)|^{2k-3} \\
%      \ll &\phi(q)(\log y)^{k^2-3k+3} \bigg(\frac{\log y}{e^j}\cdot \frac{\log y}{e^{\max \{j,l \}} }\cdot  \frac{\log y}{e^{ \{j,l,m \} }} \bigg)^{(2k-3)/2}\bigg(\frac{\log y}{e^l}\cdot \frac{\log y}{e^{m}}\cdot \frac{\log y}{e^{\max \{l,m \}}} \bigg)^{1/2} \\
%      = & \phi(q)(\log y)^{k^2} \exp\Big( -\frac{2k-3}{2}(j+\max \{j,l \}+\max \{j,l,m \})-\frac{1}{2}(l+m+\max \{l,m \} ) \Big). \\
% \end{align*}
% The volume is
% $$\ll B \frac{e^{j+l+m} }{(\log y)^3}$$
We have $ |\mathcal{B}_j|\ll \frac{e^j}{\log y}$, so in total we obtain 
\begin{align*}
       &  \sum_{\substack{1\leq j<K \\ 1\leq l,m \leq K}} j^{2k-3} |\mathcal{B}_j|^{2k-4} \int_{\mathcal{C}_{j,l,m}} L(\mathbf{t},u)du d\mathbf{t} \\
    \ll & \phi(q)(\log y)^{(k-1)^2} B^3(\log\log B)^{O(1)} \cdot \\
    &  \cdot \sum_{\substack{1\leq j<K \\ 1\leq l,m \leq K}} j^{2k-3} \exp\Big( j(k-3/2)+l/2+m/2-(k-3/2)( \max \{j,l \}+\max \{j,l,m \})-\max\{l,m \}/2\Big) \\
    \ll &  \phi(q)(\log y)^{(k-1)^2} B^3(\log\log B)^{O(1)}\sum_{1\leq j,l,m\leq K} j^{2k-3} e^{-(k-2)\max \{j,l,m \}} \\
    \ll &  \phi(q)(\log y)^{(k-1)^2} B^3(\log\log B)^{O(1)} \\
\end{align*}
\textbf{Case 2} $j=K$. The volume of the region $\mathcal{C}_{K,l,m}$ is $\ll B^4\frac{e^{l+m} }{(\log y)^2}$. For each $i=1,2,3$ we have $g^*(|t_i-u|)\ll \log\log  B$, $g^*(|t_1-t_2|)\ll\frac{\log y}{e^l} \log\log  B$,  $g^*(|t_2-t_3|)\ll\frac{\log y}{e^m} \log\log  B$,  $g^*(|t_1-t_3|)\ll\frac{\log y}{e^l }\log\log  B$, so
$$  L(\mathbf{t},u) \ll  \phi(q) (\log y)^{k^2-3k+9/2}(\log \log B)^{O(1)}e^{-l-m/2}$$ 
Finally, $|\mathcal{B}_K|\ll B$, so
\begin{align*}
     \sum_{1\leq  l, m\leq K} K^{2k-3} |\mathcal{B}_K|^{2k-4} \int_{\mathcal{C}_{K,l,m}} L(\mathbf{t},u)du d\mathbf{t}\ll & \phi(q) (\log y)^{k^2-3k+5/2}B^{2k}(\log \log B)^{O(1)}(\log\log y)^{O(1)}\sum_{1\leq  l, m\leq K} e^{m/2} \\
     \ll &  \phi(q) (\log y)^{k^2-3k+3}B^{2k}(\log \log B)^{O(1)}(\log \log y)^{O(1)} \\
\end{align*}
\end{proof}
As stated in Theorem 3, when $y\leq q^{1/2}$ we would like to show that 
\begin{equation}
\label{theorem3firstpart}
    \sum_{\chi\in X_q^*} \bigg|\sum_{n\leq y}\chi(n)\bigg|^{2k}\ll y^k\phi(q) (\log y)^{(k-1)^2}.
\end{equation}
We first consider a weighted version of this moment. We choose the weight to be piecewise linear and continuous. This ensures that the corresponding Perron integral is absolutely convergent.
\begin{lemma}
\label{fweight}
Let $2\leq y\leq q$ and $T=y/(\log y)^C$, where $C>0$ is a fixed constant.
Let $f(x)=1$ if $0<x\leq y-T$, $f(x)=1-\frac{x-y+T}{T}$ if $y-T\leq x\leq y$, and $f(x)=0$ if $x\geq y$. One has
\begin{equation}
\label{theorem3firstweighted}
\sum_{\chi\in X_q^*} \bigg|\sum_{n\leq y}f(n)\chi(n)\bigg|^{2k}\ll y^k\phi(q) (\log y)^{(k-1)^2}.
\end{equation}
\end{lemma}
\begin{proof}
We have
$$\int_{0}^{\infty} f(x)x^s \frac{dx}{x}=\frac{1}{Ts(s+1)}(y^{s+1}-(y-T)^{s+1}),$$
so by Perron's formula (or the Mellin inversion formula), for each $c>1$ one has
$$\sum_{n\leq y}f(n)\chi(n)=\frac{1}{2\pi i}\int_{c-i\infty}^{c+i\infty} \frac{L(s,\chi)}{Ts(s+1)}(y^{s+1}-(y-T)^{s+1})ds,$$
since the integral is absolutely convergent. We may shift the line of integration to the left and choose $c=1/2+1/\log y$. Also, by symmetry it is enough to consider the integral in the upper half-plane. Let $A:=(\log y)^D$, where $D$ is a large constant specified later. We break up the resulting integral into three pieces according to the size of $t:=\Im s$. In fact, our three regions of integration will be $[0,A]$, $[A,q]$ and $[q,\infty]$. When $0\leq t\leq A$ and $\Re s=1/2+1/\log y$, we use the estimate $|y^{s+1}-(y-T)^{s+1}|\ll |(s+1)y^{s}T|\ll T(t+1)y^{1/2}$, when $t>A$ we use $|y^{s+1}-(y-T)^{s+1}|\ll y^{3/2}$ and recall that $T=y/(\log y)^C$. We get
\begin{equation}
\label{breakup}
    \begin{split}
    \sum_{\chi \in X_q^*} \bigg|\sum_{n\leq y}f(n)\chi(n)\bigg|^{2k}\ll & y^k\sum_{\chi \in X_q^*}\bigg(\int_0^A \frac{|L(1/2+1/\log y +it,\chi)|}{t+1}dt\bigg)^{2k} \\  
    +& y^k(\log y)^{2kC}\sum_{\chi \in X_q^*} \bigg( \int_{A}^q \frac{|L(1/2+1/\log y+it,\chi) |}{t^2}  dt \bigg)^{2k} \\
    + & y^k(\log y)^{2kC} \sum_{\chi \in X_q^*} \bigg( \int_{q}^{\infty} \frac{|L(1/2+1/\log y+it,\chi) |}{t^2} dt \bigg)^{2k}. \\
\end{split}
\end{equation}
% We would like to show
% $$\sum_{\chi \in X_q^*}\bigg( \int_{0}^A \frac{|L(1/2+1/\log y+it ,\chi) |}{t+1} dt \bigg)^{2k} \ll \phi(q)(\log y)^{(k-1)^2}$$
Let us deal with the integrals ranging from $0$ to $A$. We would like to use the decay of $1/(t+1)$ efficiently, so we break up each integral via Hölder's inequality. For each $\chi\in X_q^*$ we obtain
\begin{equation*}
\begin{split}
    \bigg(\int_0^A \frac{|L(1/2+1/\log y +it,\chi)|}{t+1}dt\bigg)^{2k}&\leq \bigg(\sum_{n\leq \log A+1} n^{-2k/(2k-1)} \bigg)^{2k-1}\sum_{n\leq \log A+1} \bigg(n\int_{e^{n-1}-1}^{e^{n}-1 } \frac{|L(1/2+1/\log y+it,\chi) |}{t+1} dt\bigg)^{2k}   \\
    &\ll \sum_{n\leq \log A+1} \frac{n^{2k} }{e^{2nk} } \bigg( \int_{e^{n-1}-1}^{e^{n}-1 } |L(1/2+1/\log y+it,\chi) | dt \bigg)^{2k}.
\end{split}
\end{equation*}
Using this inequality and Proposition \ref{t3prop}, we obtain
\begin{align*}
  & \sum_{\chi \in X_q^*} \bigg( \int_{0}^A \frac{|L(1/2+1/\log y+it,\chi) |}{t+1} dt \bigg)^{2k}  \\ 
   \ll &  \sum_{n\leq \log A+2} \frac{n^{2k} }{e^{2nk}} \sum_{\chi \in X_q^*} \bigg( \int_{e^{n-1}-1}^{e^{n}-1 } |L(1/2+1/\log y+it,\chi) | dt \bigg)^{2k} \\
  \ll & \phi(q)\sum_{n\leq \log A+2}\frac{n^{2k}}{e^{2nk} } \Big((\log y)^{(k-1)^2} e^{3n}(\log 2n)^{O(1)} + (\log y)^{k^2-3k+3}e^{2nk} (\log 2n)^{O(1)}(\log\log y)^{O(1)}\Big) \\
  \ll & \phi(q)(\log y)^{(k-1)^2}. \\
\end{align*}
Note that we used $2k-3>0$ and that $k^2-3k+3<(k-1)^2$ in the last line.

Next we deal with the integrals ranging from $A$ to $q$. The argument is the same as in the previous case. We use Hölder's inequality and Proposition \ref{t3prop} to get the upper bound
\begin{align*}
  & \sum_{\chi \in X_q^*} \bigg( \int_{A}^q \frac{|L(1/2+1/\log y+it,\chi) |}{t^2}  dt \bigg)^{2k}  \\ 
   \ll & \sum_{\log A-1\leq n\leq \log q} \frac{n^{2k} }{e^{4nk}} \sum_{\chi \in X_q^*} \bigg( \int_{e^{n}}^{e^{n+1} } |L(1/2+1/\log y+it,\chi) | dt \bigg)^{2k} \\
  \ll &  \sum_{\log A-1\leq n \leq \log q}\frac{n^{2k}}{e^{4nk} }\cdot \phi(q) \Big((\log y)^{(k-1)^2} e^{3n} (\log 2n)^{O(1)}+e^{2nk} (\log 2n)^{O(1)} (\log\log y)^{O(1)}(\log y)^{k^2-3k+3}\Big)  \\
  \ll & \phi(q)(\log y)^{(k-1)^2-D}. \\
\end{align*}
Using the pointwise bound $|L(1/2+1/\log y+it,\chi)|\ll (qt)^{1/2}$  when $t\geq q$ (see Lemma 10.15 in \cite{montgomery2007multiplicative}), we get
$$ \sum_{\chi \in X_q^*} \bigg( \int_{q}^{\infty} \frac{|L(1/2+1/\log y+it,\chi) |}{t^2} dt \bigg)^{2k} \ll \phi(q)$$
Choosing $D$ large enough in terms of $C$ and $k$, we obtain (\ref{theorem3firstweighted}).
\end{proof}
Next, we show that the weighted moment is close to the unweighted one. 

\begin{lemma}
\label{fdiff}
Let $f$ be the function defined in Lemma \ref{fweight} with some fixed $C>0$. We then have
    \begin{equation}
\label{theorem3firstrest}
\sum_{\chi\in X_q^*} \bigg|\sum_{n\leq y}(1-f(n))\chi(n)\bigg|^{2k}\ll \phi(q)y^k(\log y)^{2k^2-C/2}
\end{equation}
\end{lemma}
\begin{proof}
By the Cauchy-Schwarz inequality
\begin{equation}
\label{pocs1}
    \sum_{\chi\in X_q^*} \bigg|\sum_{n\leq y}(1-f(n))\chi(n)\bigg|^{2k}\leq \bigg(\sum_{\chi\in X_q^*} \bigg|\sum_{n\leq y}(1-f(n))\chi(n)\bigg|^2\bigg)^{1/2}\bigg(\sum_{\chi\in X_q^*} \bigg|\sum_{n\leq y}(1-f(n))\chi(n)\bigg|^{4k-2}\bigg)^{1/2}.
\end{equation}
Firstly, using the orthogonality of characters, we get
\begin{equation}
\label{pocs2}
    \sum_{\chi\in X_q^*} \bigg|\sum_{n\leq y}(1-f(n))\chi(n)\bigg|^2\leq \phi(q) \sum_{n\leq y}|1-f(n)|^2\leq \phi(q)y(\log y)^{-C}.
\end{equation}
Secondly,
\begin{equation}
\label{pocs4}
\sum_{\chi\in X_q^*} \bigg|\sum_{n\leq y}(1-f(n))\chi(n)\bigg|^{4k-2}\ll \sum_{\chi\in X_q^*} \bigg|\sum_{n\leq y}\chi(n)\bigg|^{4k-2}+\sum_{\chi\in X_q^*} \bigg|\sum_{n\leq y}f(n)\chi(n)\bigg|^{4k-2}
\end{equation}
We only need to estimate the first sum on RHS as the second one is already dealt with. By Perron's formula
\begin{align*}
    2\pi i \sum_{n\leq y}\chi(n)= & \int_{1+1/\log y-iy}^{1+1/\log y+iy}L(s,\chi) \frac{y^s}{s}ds +O(\log^2 y)\\
     = &\int_{1/2+1/\log y-iy}^{1/2+1/\log y+iy} +\int_{1+1/\log y -iy}^{1/2+1/\log y-iy}+\int_{1/2+1/\log y+iy}^{1+1/\log y+iy} L(s,\chi)\frac{y^s}{s}ds +O(\log^2 y) \\
\end{align*}
We first address the moments of the horizontal integrals. We may assume that $y\geq 10$, otherwise the lemma is trivial. By symmetry we need to consider only one of them. We have $|y^s/s|\ll 1$ in that range, so applying Hölder's inequality we get
\begin{align*}
 \sum_{\chi \in X_q^*}\bigg| \int_{1/2+1/\log y+iy}^{1+1/\log y+iy} L(s,\chi)\frac{y^s}{s}ds\bigg|^{4k-2}\ll & \sum_{\chi \in X_q^*}\bigg( \int_{1/2+1/\log y+iy}^{1+1/\log y+iy} |L(s,\chi)| |ds| \bigg)^{4k-2} \\
 \ll &  \sum_{\chi \in X_q^*} \int_{1/2+1/\log y+iy}^{1+1/\log y+iy} |L(s,\chi)|^{4k-2} |ds| \\
 \ll & \phi(q)(\log y)^{(2k-1)^2}. \\
\end{align*}
Here in the last line we used that if $ 1/2+1/\log y\leq \Re s\leq 1+1/\log y$, then 
$$\sum_{\chi \in X_q^*}|L(s,\chi)|^{4k-2} \ll \phi(q)(\log y)^{(2k-1)^2},$$
which is a consequence of Theorem \ref{t4} since $y\geq 10$.
The vertical integral is handled by Proposition 2 and Hölder's inequality, similarly to previous cases. We have
\begin{align*}
    & \sum_{\chi \in X_q^*}\bigg|\int_{1/2+1/\log y-iy}^{1/2+1/\log y+iy}L(s,\chi) \frac{y^s}{s}ds\bigg|^{4k-2} \\
  \ll & y^{2k-1}\sum_{\chi \in X_q^*} \bigg( \int_{0}^y \frac{|L(1/2+1/\log y+it,\chi) |}{t+1} dt \bigg)^{4k-2}  \\ 
   \ll &  y^{2k-1}\sum_{n\leq \log y+2} \frac{n^{4k-2} }{e^{(4k-2)n}} \sum_{\chi \in X_q^*} \bigg( \int_{e^{n-1}-1}^{e^{n}-1 } |L(1/2+1/\log y+it,\chi) | dt \bigg)^{4k-2} \\
  \ll & y^{2k-1}\phi(q)(\log y)^{(2k-1-1)^2} \sum_{n\leq \log y+2}\frac{n^{4k-2}}{e^{(4k-2)n} } +y^k\phi(q)(\log y)^{(2k-1)^2-3(2k-1)+3}(\log\log y)^{O(1)} \sum_{n\leq \log y+2}n^{4k-2} \\
  \ll & y^{2k-1}\phi(q)(\log y)^{4k^2}. \\
\end{align*}
This implies the crude bound
\begin{equation}
\label{pocs3}
   \sum_{\chi\in X_q^*} \bigg|\sum_{n\leq y}\chi(n)\bigg|^{4k-2}\ll y^{2k-1}\phi(q)(\log y)^{4k^2} .
\end{equation}
Thus, recalling \eqref{pocs4} and using Lemma \ref{fweight} we have shown
$$\sum_{\chi\in X_q^*} \bigg|\sum_{n\leq y}(1-f(n))\chi(n)\bigg|^{4k-2}\ll y^{2k-1}\phi(q)(\log y)^{4k^2}.$$
Now we combine this inequality with \eqref{pocs1} and \eqref{pocs2} to get the lemma.
\end{proof}
Now adding up the inequalities in Lemma \ref{fweight} and \ref{fdiff} with $C=4k^2$, say, yields the first part of Theorem \ref{t3}.
\subsection{Proof of the second part of Theorem \ref{t3}}
We now show the second part of Theorem 3, that is if $y\leq q^{1/2}$, then
\begin{equation}
\label{theorem3secondpart}
    \sum_{\chi\in X_q^*} \bigg|\sum_{n\leq q/y}\chi(n)\bigg|^{2k}\ll (q/y)^k\phi(q) (\log 2y)^{(k-1)^2}.
\end{equation}
The proof is similar to that of (\ref{theorem3firstpart}), except shift the line of integration to the vertical line with real part $c=1/2-1/\log y$ and then use the inequality $|L(s,\chi)|\ll (tq)^{1/2-\sigma}|L(1-s,\bar\chi)|$, which is a consequence of the functional equation. Again, we start with a weighted version. In addition we assume that $y\geq 10$. 
We change the definition of $f$ slightly. Let $f(x)=1$ if $0<x\leq q/y-T$, $f(x)=1-\frac{x-q/y+T}{T}$ if $q/y-T\leq x\leq q/y$, and $f(x)=0$ if $x\geq q/y$, where we choose $T=\frac{q}{y(\log y)^C}$, where $C$ is a sufficiently large fixed constant. Using Perron's formula and $|L(s,\chi)|\ll (tq)^{1/2-\sigma}|L(1-s,\chi)|$, we get
\begin{align*}
    \bigg|\sum_{n\leq q/y} \chi(n)f(n)\bigg| \ll & \bigg|\int_{1/2-1/\log y-i\infty}^{1/2-1/\log y+i\infty}\frac{L(s,\chi)}{Ts(s+1)}\big( (q/y)^{s+1}-(q/y-T)^{s+1}\big)ds\bigg| \\
    \ll & \int_0^{\infty}\frac{|L(1/2-1/\log y+it,\bar\chi)|}{T(t+1)^2} \big| (q/y)^{3/2-1/\log y +it}-(q/y-T)^{3/2-1/\log y +it}\big| dt \\
    \ll &\int_0^{\infty}\frac{q^{1/\log y} |L(1/2+1/\log y+it,\bar\chi)|}{T(t+1)^{2-1/\log y} } \big| (q/y)^{3/2-1/\log y +it}-(q/y-T)^{3/2-1/\log y +it}\big| dt. \\
\end{align*}
When $t\leq A$ (recall $A=(\log y)^D$ where $D$ is a large constant) we use that $\big| (q/y)^{3/2-1/\log y +it}-(q/y-T)^{3/2-1/\log y +it}\big|\ll (q/y)^{1/2-1/ \log y} T(t+1)$, if $q>A$, then $\big| (q/y)^{3/2-1/\log y +it}-(q/y-T)^{3/2-1/\log y +it}\big|\ll (q/y)^{3/2-1/\log y}$.
\begin{align*}
    \sum_{\chi \in X_q^*} \bigg|\sum_{n\leq q/y}f(n)\chi(n)\bigg|^{2k}\ll & (q/y)^k\sum_{\chi \in X_q^*} \bigg(\int_0^A \frac{|L(1/2+1/\log y +it,\chi)|}{t+1}\bigg)^{2k} \\  
    +& (q/y)^k(\log y)^{2kC}\sum_{\chi \in X_q^*} \bigg( \int_{A}^q \frac{|L(1/2+1/\log y+it,\chi) |}{(t+1)^{2-1/\log y} }  dt \bigg)^{2k} \\
    + & (q/y)^k(\log y)^{2kC} \sum_{\chi \in X_q^*} \bigg( \int_{q}^{\infty} \frac{|L(1/2+1/\log y+it,\chi) |}{(t+1)^{2-1/\log y}} dt \bigg)^{2k} \\
\end{align*}
This can be bounded the same way as it is done in (\ref{breakup}), the only difference is that the exponent of $t+1$ in the denominator is $2-1/\log y\geq 3/2$ instead of 2, since $y\geq 10$, but the argument can be run the same way, so we have shown that if $y\geq 10$, then
\begin{equation}
    \label{hello}
\sum_{\chi \in X_q^*} \bigg|\sum_{n\leq q/y}f(n)\chi(n)\bigg|^{2k}\ll \phi(q)(q/y)^k (\log 2y)^{(k-1)^2}
\end{equation}
Our remaining task is to show \eqref{hello} without the weights $f(n)$ and to handle the case $y\leq 10$. It turns out that we can do these by proving the following inequality.
\begin{proposition}
\label{prop4}
   For all $1\leq y\leq q^{1/2}$ we have
\begin{equation}
    \label{montgomery}
    \sum_{\chi \in X_q^*} \bigg|\sum_{n\leq q/y}\chi(n)\bigg|^{2k}\ll \phi(q) (q/y)^k (\log 2y)^{O(1)}.
\end{equation}
\end{proposition}
\begin{proof}[Proof that Proposition \ref{prop4} and \eqref{hello} shows 
\eqref{theorem3secondpart}]

 For $y\leq 10$ the inequality (\ref{montgomery}) is the same as (\ref{theorem3secondpart}), on the other hand when $y\geq 10$, \eqref{montgomery} implies (if $C$ is large enough)
$$\sum_{\chi\in X_q^*} \bigg|\sum_{n\leq q/y}(1-f(n))\chi(n)\bigg|^{2k}\ll \phi(q)(q/y)^k(\log 2y)^{(k-1)^2}$$
the same way as $(\ref{pocs3})$ implied (\ref{theorem3firstrest}). This last inequality with (\ref{hello}) then implies (\ref{theorem3secondpart}).
\end{proof}
\begin{proof}[Proof of Proposition \ref{prop4}]
By Hölder we may assume that $k$ is an integer. Our argument is inspired by the proof of Theorem 1 of \cite{montgomery1979mean}. By Lemma 1 of \cite{montgomery1979mean} we have for any $H>1$,
\begin{equation}
    \sum_{n\leq q/y}\chi(n)=\frac{-\tau(\chi) }{2\pi i}\sum_{1\leq |h|\leq H}\frac{\bar\chi(h)}{h}\big(e(h/y)-1\big)+O(1+qH^{-1}\log q).
\end{equation}
Here, $\tau(\chi)=\sum_{n=1}^q \chi(n)e\big(\frac{n}{q}\big)$ is the usual Gauss sum and $e(x)=e^{2\pi i x}$. For $\chi$ primitive we have $|\tau(\chi)|=q^{1/2}$, so
$$\sum_{\chi \in X_q^*} \bigg|\sum_{n\leq q/y}\chi(n)\bigg|^{2k}\ll \phi(q)(1+(qH^{-1}\log q))^{2k}+q^k\sum_{\chi \in X_q}\bigg|\sum_{1\leq h\leq H}\frac{\chi(h)}{h}\big(e(h/y)-1\big)\bigg|^{2k}.$$
We thus choose $H=q$ to get an acceptable contribution from the error terms. Let $a_h=\frac{e(h/y)-1}{h}$, then $a_h\ll \min(1/y, 1/h)$, so
$$\bigg|\sum_{1\leq h\leq H}a_h\chi(h)\bigg|^{k}=\sum_{h\leq H^k} b_h\chi(h), $$
where $b_h\ll d_k(h) \min(1/y^k, 1/h)$, so by the orthogonality of characters
\begin{equation}
    \label{orthogonal}
\sum_{\chi \in X_q}\bigg|\sum_{1\leq h\leq H}\frac{\chi(h)}{h}\big(e(h/y)-1\big)\bigg|^{2k}=\sum_{\chi \in X_q}\bigg|\sum_{1\leq h\leq H^k}b_h\chi(h)\bigg|^{2}= \phi(q)\sum_{\substack{1\leq m\leq q \\ (m,q)=1} } \bigg|\sum_{\substack{1\leq h\leq H^k \\ h\equiv m \, (q)}} b_h \bigg|^2 
\end{equation}
For each $1\leq m\leq q$ one has
$$\sum_{\substack{q\leq h\leq H^k \\ h\equiv m \, (q)}} b_h\ll \frac{\max_{1\leq h\leq H^k} d_k(h)}{q}\sum_{h\leq H^K} \frac{1}{h}\ll q^{-1+o(1)}.$$
Therefore, if in \eqref{orthogonal} inside each inner sum we separate whether $h\leq q$ or $h>q$ we get that \ref{orthogonal} is
$$\ll \phi(q)\sum_{1\leq h\leq q}d_k(h)^2 \min (y^{-2k}, h^{-2})+q^{o(1)}.$$
By (10) of \cite{montgomery1979mean} we have
$$\sum_{1\leq h\leq q}d_k(h)^2 \min (y^{-2k}, h^{-2})\leq y^{-2k}\sum_{h\leq y^k} d_k(h)^2+\sum_{h\geq y^k} \frac{d_k(h)^2}{h^2}\ll y^{-k}(\log 2y)^{k^2-1},$$
so we gain that if $k$ is an integer then
$$\sum_{\chi \in X_q^*} \bigg|\sum_{n\leq q/y}\chi(n)\bigg|^{2k}\ll \phi(q)(q/y)^{k}(\log 2y)^{k^2-1}+q^{k+o(1) } .$$
This shows (\ref{montgomery}) when $y\leq q^{1/2k}$. On the other hand, when $y> q^{1/2k}$ the inequality (\ref{montgomery}) is in fact implied by (\ref{pocs3}) (note that this is the part where we use GRH for the proof of this proposition). Therefore the second part of Theorem \ref{t3} is proved.
\end{proof}
\section{Acknowledgements}
The author is supported by the Warwick Mathematics Institute Centre for Doctoral Training, and gratefully acknowledges funding from the University of Warwick and the UK Engineering and Physical Sciences Research Council (Grant number: EP/TS1794X/1). The author is grateful to Adam Harper for the many useful discussions, suggestions and for carefully reading through an earlier version of this manuscript.
\printbibliography

@article{sound,
  title={Moments of the Riemann zeta function},
  author={Soundararajan, Kannan},
  journal={Annals of Mathematics},
  pages={981--993},
  year={2009},
  publisher={JSTOR}
}

@article{munschshifted,
  title={Shifted moments of $ L $-functions and moments of theta functions},
  author={Munsch, Marc},
  journal={Mathematika},
  volume={63},
  number={1},
  pages={196--212},
  year={2017},
  publisher={London Mathematical Society}
}

@article{harpersharp,
  title={Sharp conditional bounds for moments of the Riemann zeta function},
  author={Harper, Adam J},
  journal={arXiv preprint arXiv:1305.4618},
  year={2013}
}

@book{montgomery2007multiplicative,
  title={Multiplicative number theory I: Classical theory},
  author={Montgomery, Hugh L and Vaughan, Robert C},
  number={97},
  year={2007},
  publisher={Cambridge university press}
}

@article{montgomery1979mean,
  title={Mean values of character sums},
  author={Montgomery, Hugh L and Vaughan, Robert C},
  journal={Canadian Journal of Mathematics},
  volume={31},
  number={3},
  pages={476--487},
  year={1979},
  publisher={Cambridge University Press}
}

@article{harper2020moments,
  title={Moments of random multiplicative functions, II: High moments},
  author={Harper, Adam J},
  journal={Algebra \& Number Theory},
  volume={13},
  number={10},
  pages={2277--2321},
  year={2020},
  publisher={Mathematical Sciences Publishers}
}

@article{pohlya1918ugber,
  title={Uber die Verteilung der quadratischen Reste und Nichtreste},
  author={Pólya, G},
  journal={Gottinger Nachr},
  volume={21},
  pages={29},
  year={1918}
}

@article{louboutin2013second,
  title={The second and fourth moments of theta functions at their central point},
  author={Louboutin, St{\'e}phane R and Munsch, Marc},
  journal={Journal of Number Theory},
  volume={133},
  number={4},
  pages={1186--1193},
  year={2013},
  publisher={Elsevier}
}

@article{munsch2016upper,
  title={Upper and lower bounds for higher moments of theta functions},
  author={Munsch, Marc and Shparlinski, Igor E},
  journal={The Quarterly Journal of Mathematics},
  volume={67},
  number={1},
  pages={53--73},
  year={2016},
  publisher={Oxford University Press}
}

@inproceedings{harper2020moments1,
  title={Moments of random multiplicative functions, I: Low moments, better than squareroot cancellation, and critical multiplicative chaos},
  author={Harper, Adam J},
  booktitle={Forum of Mathematics, Pi},
  volume={8},
  year={2020},
  organization={Cambridge University Press}
}

@article{radziwill2013continuous,
  title={Continuous lower bounds for moments of zeta and $L$-functions},
  author={Radziwi{\l}{\l}, Maksym and Soundararajan, Kannan},
  journal={Mathematika},
  volume={59},
  number={1},
  pages={119--128},
  year={2013},
  publisher={London Mathematical Society}
}

@article{harper2023typical,
  title={The typical size of character and zeta sums is $o(\sqrt{x})$},
  author={Harper, Adam J},
  journal={arXiv preprint arXiv:2301.04390},
  year={2023}
}

@article{curran2023correlation,
  title={Correlations of the Riemann zeta function},
  author={Curran, Michael J},
  journal={arXiv preprint arXiv:2303.10123},
  year={2023}
}

@article{ng2022shifted,
  title={Shifted moments of the Riemann zeta function},
  author={Ng, Nathan and Shen, Quanli and Wong, Peng-Jie},
  journal={arXiv preprint arXiv:2206.03350},
  year={2022}
}

@article{chandee2011correlation,
  title={On the correlation of shifted values of the Riemann zeta function},
  author={Chandee, Vorrapan},
  journal={Quarterly journal of mathematics},
  volume={62},
  number={3},
  pages={545--572},
  year={2011},
  publisher={OUP}
}

@article{cochrane1998high,
  title={High order moments of character sums},
  author={Cochrane, Todd and Zheng, Zhiyong},
  journal={Proceedings of the American Mathematical Society},
  volume={126},
  number={4},
  pages={951--956},
  year={1998}
}
\end{document}